\documentclass[12pt]{amsart}  %[12pt] while developing, omitted for final
\usepackage[latin1]{inputenc}
\usepackage{amsmath} 
\usepackage{amsfonts}
\usepackage{amssymb}
\usepackage{stmaryrd}
\usepackage{latexsym} 
\usepackage{graphicx}
\usepackage{subfigure}
\usepackage{color}
\usepackage{hyperref}
\usepackage{verbatim}
\usepackage[all]{xy}
\usepackage{graphics}
\usepackage{pdfsync}
\theoremstyle{plain}
\newtheorem{theorem}{Theorem}[section]
\newtheorem{proposition}[theorem]{Proposition}

\newtheorem{lemma}[theorem]{Lemma}
\newtheorem{corollary}[theorem]{Corollary}
\newtheorem{conjecture}[theorem]{Conjecture}
\newtheorem{definition}[theorem]{Definition}
\newtheorem{example}[theorem]{Example}

\theoremstyle{remark}

\numberwithin{equation}{section}
\newcommand{\bC}{\mathbb{C}}

\newcommand{\coverc}{\lessdot _{{c}}}

\newcommand{\lessc}{< _{{c}}}

\newcommand{\Lc}{\mathcal{L}_{{c}}}

\newcommand{\po}{\preccurlyeq _\alpha}
\newcommand{\equiva}{\sim _\alpha}

\newcommand{\tcr}[1]{\textcolor{red}{#1}}

\newcommand{\Sym}{\ensuremath{\operatorname{Sym}}}

\newcommand{\Qsym}{\ensuremath{\operatorname{QSym}}}
\newcommand{\qs}{{\mathcal{S}}}		% quasisymmetric  schurs
\newcommand{\SRCT}{\ensuremath{\operatorname{SRCT}}}

\newcommand{\Nsym}{\ensuremath{\operatorname{NSym}}}
\newcommand{\sgrp}{\mathfrak{S}}

\newcommand{\hn}{H_n(0)}
\newcommand{\smodule}{\mathbf{S}}

\newcommand\partitionof[1]{\widetilde{#1}}

\newcommand{\set}{\mathrm{set}} % for sets corresponding to compositions
\newcommand{\des}{\mathrm{Des}} % for descent sets 
\newcommand{\comp}{\mathrm{comp}} % for compositions corresponding to sets
\newcommand{\col}{\mathrm{col}} % for column sequence
\newcommand{\tup}{\mathfrak{t}}
\newcommand{\rtau}{{\tau}} % for reverse composition tableaux

\newcommand{\suchthat}{\;|\;}
\newcommand{\cskew}{{/\!\!/}}% for skew composition shapes
\newcommand{\spam}{\operatorname{span}}

\newcommand{\st}{\ensuremath{\operatorname{st}}}

\newlength\cellsize 
\savebox2{%
\begin{picture}(20,20)
\put(0,0){\line(1,0){20}}
\put(0,0){\line(0,1){20}}
\put(20,0){\line(0,1){20}}
\put(0,20){\line(1,0){20}}
\end{picture}}
\newcommand\cellify[1]{\def\thearg{#1}\def\nothing{}%
\ifx\thearg\nothing
\vrule width0pt height\cellsize depth0pt\else
\hbox to 0pt{\usebox2\hss}\fi%
\vbox to 20\unitlength{
\vss
\hbox to 20\unitlength{\hss$#1$\hss}
\vss}}
\newcommand\tableau[1]{\vtop{\let\\=\cr
\setlength\baselineskip{-16000pt}
\setlength\lineskiplimit{16000pt}
\setlength\lineskip{0pt}
\halign{&\cellify{##}\cr#1\crcr}}}
\savebox3{%
\begin{picture}(20,20)
\put(0,0){\line(1,0){20}}
\put(0,0){\line(0,1){20}}
\put(20,0){\line(0,1){20}}
\put(0,20){\line(1,0){20}}
\end{picture}}
\newcommand\expath[1]{%
\hbox to 0pt{\usebox3\hss}%
\vbox to 20\unitlength{
\vss
\hbox to 20\unitlength{\hss$#1$\hss}
\vss}}

\newcommand\bas[1]{\omit \vbox to \cellsize{ \vss \hbox to \cellsize{\hss$#1$\hss} \vss}}

\begin{document}

\title[$0$-Hecke algebras and quasisymmetric Schur functions]{Modules of the $0$-Hecke algebra and quasisymmetric Schur functions}

\author{Vasu V. Tewari}
\address{Department of Mathematics, University of British Columbia, Vancouver, BC V6T 1Z2, Canada}
\email{\href{mailto:vasu@math.ubc.ca}{vasu@math.ubc.ca}}
\author{Stephanie J. van Willigenburg}
\address{Department of Mathematics, University of British Columbia, Vancouver, BC V6T 1Z2, Canada}
\email{\href{mailto:steph@math.ubc.ca}{steph@math.ubc.ca}}

\thanks{
The authors were supported in part by the National Sciences and Engineering Research Council of Canada.}

\begin{abstract} 
We begin by deriving an action of the $0$-Hecke algebra on standard reverse composition tableaux and use it to discover $0$-Hecke modules whose quasisymmetric characteristics are the natural refinements of Schur functions known as quasisymmetric Schur functions. Furthermore, we classify combinatorially which of these $0$-Hecke modules are indecomposable. 

{}From here, we establish that the natural equivalence relation arising from our $0$-Hecke action has equivalence classes that are isomorphic to subintervals of the weak Bruhat order on the symmetric group. Focussing on the equivalence classes containing a canonical tableau we discover a new basis for the Hopf algebra of quasisymmetric functions, and use the cardinality of these equivalence classes to establish new enumerative results on truncated shifted reverse tableau studied by Panova and Adin-King-Roichman.

 Generalizing our $0$-Hecke action to one on skew standard reverse composition tableaux, we derive $0$-Hecke modules whose quasisymmetric characteristics are the skew quasisymmetric Schur functions of Bessenrodt et al. This enables us to prove a restriction rule that reflects the coproduct formula for quasisymmetric Schur functions, which in turn yields a quasisymmetric branching rule analogous to the classical branching rule for Schur functions.

\end{abstract}

\maketitle
\tableofcontents

\section{Introduction}\label{sec:intro} Dating from an 1815 paper of Cauchy \cite{cauchy}, Schur functions have long been a central object of study, due to their multi-faceted nature. Starting from their definition, they can be viewed as the determinant of a matrix, can be computed using divided differences, expressed using raising operators, or as a sum of combinatorially computed monomials. Additionally, they arise in a number of different guises in mathematics including as an orthonormal basis for the Hopf algebra of symmetric functions $\Sym$, as characters of the irreducible polynomial representations of $GL(n, \bC)$ from where the name derives \cite{schur}, as generating functions for semistandard Young tableaux, and in particular, in the representation theory of the symmetric group as the image of the irreducible characters under the characteristic map. These facets and more can be explored more deeply in texts such as \cite{macdonald-1, sagan, stanley-ec2}.

Their pervasive nature has also led to their generalization in a variety of ways including Schur $P$ functions arising in the representation theory of the
double cover of the
symmetric group \cite{ macdonald-1, stembridge-1}, $k$-Schur functions connected to the enumeration of Gromov-Witten invariants \cite{lapointe-morse}, shifted Schur functions related to the representation theory of the general linear group \cite{okounkov-1}, and factorial Schur functions that are special cases of double Schubert polynomials \cite{lascoux-1,molev-sagan}. Perhaps the most well-known generalization is Macdonald polynomials \cite{macdonald-polys}, which encompass a plethora of functions and have additional parameters $q$ and $t$ that reduce to Schur functions when $q=t=0$. Likewise, the self-dual Hopf algebra $\Sym$ has been generalized in a number of ways. One significant nonsymmetric generalization is the Hopf algebra of quasisymmetric functions \cite{gessel}, $\Qsym$, that is Hopf dual to the Hopf algebra of noncommutative symmetric functions \cite{GKLLRT}, $\Nsym$, which is itself a noncommutative generalization of $\Sym$.

Although much more recent than $\Sym$, both $\Qsym$ of which $\Sym$ is a subalgebra, and $\Nsym$ of which $\Sym$ is a quotient, are of interest. For example, $\Nsym$ is isomorphic \cite{GKLLRT} to Solomon's descent algebra of Type $A$ \cite{garsia-reut, solomon} and $\Qsym$ is isomorphic to the Hopf algebra of ladders \cite{foissy} and the algebra of smooth functions on Young diagrams \cite{AFNT}. Furthermore, quasisymmetric functions can be seen as generating functions for $P$-partitions \cite{gessel}, flags in graded posets \cite{ehrenborg-1} and matroids \cite{billera-jia-reiner, derksen, luoto} when considering discrete geometry and enumerative combinatorics. Considering representation theory, quasisymmetric functions arise in the study of Hecke algebras \cite{hivert}, Lie representations \cite{gessel-reut}, and crystal graphs for general linear Lie superalgebras \cite{kwon}. In probability they arise in the investigation of random walks \cite{hersh-hsiao} and riffle shuffles \cite{stanley-riffle}. They have also been seen to simplify many symmetric functions such as Macdonald polynomials \cite{HHL-1, HLMvW-1}, Kazhdan-Lusztig polynomials \cite{billera-brenti} and identify equal ribbon Schur functions \cite{HDL}, as well as arising in category theory \cite{aguiar-bergeron-sottile}, topology \cite{baker-richter, gnedin-olshanski}, and graph theory \cite{chow, humpert, stanley-chrom}.

Since $\Sym$ is a Hopf subalgebra of $\Qsym$ a desirable basis of $\Qsym$ would be one that refines and reflects the properties of Schur functions. Until recently this was often considered to be the basis of fundamental quasisymmetric functions since in the representation theory of the $0$-Hecke algebra they are the image of irreducible module isomorphism classes under the quasisymmetric characteristic map. However, recently a new basis whose genesis is in the combinatorics of Macdonald polynomials \cite{HHL-1} was discovered \cite{HLMvW-1}. This basis refines many properties of Schur functions including exhibiting a quasisymmetric Pieri rule \cite{HLMvW-1}, and quasisymmetric Littlewood-Richardson rule \cite{BLvW}, and generating quasisymmetric Kostka numbers \cite{HLMvW-1}. It was also key in resolving a long-standing conjecture of F. Bergeron and Reutenauer \cite{lauve-mason} and has initiated a new avenue of research for other Schur-like bases, which to date include row strict quasisymmetric functions \cite{ferreira, mason-remmel}, dual immaculate quasisymmetric functions \cite{BBSSZ}, Young quasisymmetric functions \cite{LMvW}, and affine quasisymmetric Schur functions \cite{berg-serrano}. Nevertheless, the important question still remains whether quasisymmetric Schur functions are the image of module isomorphism classes under the quasisymmetric characteristic map. In this paper we will answer this question, plus others, in the affirmative.

More precisely, this paper is structured as follows. In Section~\ref{sec:background} we review the necessary background required to state and prove our results. Then in Section~\ref{sec:0Heckeaction} we define operators on standard reverse composition tableaux before proving they are a $0$-Hecke action (Theorem~\ref{the:0heckerels}). These operators induce a partial order on our tableaux and Section~\ref{sec:partialorder} is devoted to studying this partial order. Extending this partial order to a total order in Section~\ref{sec:0Hecke} we are able to create $0$-Hecke modules whose isomorphism classes under the quasisymmetric characteristic map are quasisymmetric Schur functions (Theorem~\ref{the:bigone}). In Section~\ref{sec:sourcesink} we partition our standard reverse composition tableaux of a fixed shape into equivalence classes and discover that each equivalence class, when ordered using the partial order from Section~\ref{sec:partialorder}, is, in fact, isomorphic to a subinterval of the weak Bruhat order on the symmetric group (Theorem~\ref{the:bruhatordersubinterval}). The equivalence relation is applied again in Section~\ref{sec:cyclic_and_indecomposable_modules} where we combinatorially classify those modules that are indecomposable (Theorem~\ref{the:indecomposable}). In Section~\ref{sec:canonical_basis} we have two further applications. The first is that we discover a new basis for $\Qsym$ called the canonical basis, and the second is the enumeration of truncated shifted tableaux of certain shape. Section~\ref{sec:restriction} establishes that skew quasisymmetric Schur functions also arise as quasisymmetric characteristics of $0$-Hecke modules (Theorem~\ref{the:skewbigone}) and uses this to develop a restriction rule (Theorem~\ref{the:restriction_rule}) and a branching rule (Corollary~\ref{cor:branching_rule}) similar to that of Schur functions. Finally, we conclude with some future directions to pursue.

\subsection*{Acknowledgements}\label{subsec:ack} The authors would like to thank Vic Reiner and Jia Huang for helpful conversations, and the referees for thoughtful and helpful comments. They would also like to thank the referees of FPSAC 2014 whose comments on the accompanying extended abstract are also incorporated here.

\section{Background}\label{sec:background}

\subsection{Composition diagrams and composition tableaux}\label{subsec:compositions}

We begin with the classical concept of composition before moving to the contemporary concept of composition diagram. We define a \emph{composition} $\alpha=(\alpha _1, \ldots , \alpha _k)$ of $n$, denoted by $\alpha \vDash n$, to be an ordered list of positive integers whose sum is $n$. We call the $\alpha_i$ for $1\leq i \leq k$ the \emph{parts} of $\alpha$, call $k=:\ell(\alpha)$ the \emph{length} of $\alpha$ and call $n=:|\alpha |$ the \emph{size} of $\alpha$. For convenience we define the empty composition $\varnothing$ to be the unique composition of size and length $0$. 

There exists a natural bijection between compositions of $n$ and subsets of $[n-1]$ that can be described as follows. Given $S=\{i_1<\cdots <i_k\}\subseteq [n-1]$, we associate to it the composition $\alpha=(i_1,i_2-i_1,\ldots,n-i_k)$. We will refer to $\alpha$ as $\comp(S)$. Conversely, given a composition $\alpha=(\alpha_1,\ldots,\alpha_k) \vDash n$, we associate to it the subset of $[n-1]$ given by $\{\alpha_1,\alpha_1+\alpha_2,\ldots,\alpha_1+\cdots+\alpha_{k-1}\}$. We will denote this subset by $\set(\alpha)$.

Given a composition $\alpha=(\alpha _1, \ldots , \alpha _k)$, we define its \emph{reverse composition diagram}, also denoted by $\alpha$, to be the array of left-justified cells with $\alpha _i$ cells in row $i$ from the top. A cell is said to be in row $i$ if it is in the $i$-th row from the top, and in column $j$ if it is in the $j$-th column from the left, and is referred to by the pair $(i,j)$. 

Moreover, we say a part $\alpha _i$ with $1\leq i \leq k$ \emph{has a removable node} if either $i=1$ or if $i\geq 2$ then ($\alpha _i \geq 2$ and there exists no $1\leq j <i$ so $\alpha _j = \alpha _i -1$). In terms of reverse composition diagrams this is equivalent to saying row $i$ is either the top row or has no row above it containing exactly one fewer cell. If $\alpha _i = j$ then we call the cell $(i,j)$ a \emph{removable node}.

\begin{example}\label{ex:3432} The composition $\alpha = (3,4,3,2)\vDash 12$ has $\set (\alpha ) = \{3,7,10\} \subset [11]$. The reverse composition diagram of $\alpha$ is below.
$$\tableau{\ &\ &\bullet \\\ &\ &\ &\ \\\ &\ &\bullet  \\\ &\bullet }$$Note the first, third and fourth parts each have a removable node, and the removable nodes are indicated by a $\bullet$ above.
\end{example}

\begin{definition}\cite[Definition 4.1]{HLMvW-1}\label{def:SRCT} Given a composition $\alpha\vDash n$ we define a \emph{standard reverse composition tableau} (abbreviated to \emph{SRCT}) $\rtau$ of \emph{shape} $\alpha$ and \emph{size} $n$ to be a bijective filling 
$$\rtau: \alpha \rightarrow \{1, \ldots , n\}$$of the cells $(i,j)$ of the reverse composition diagram $\alpha$ such that
\begin{enumerate}
\item the entries in each row are decreasing when read from left to right
\item the entries in the first column are increasing when read from top to bottom
\item \emph{triple rule:} If $i<j$ and $\rtau (i,k)> \rtau (j, k+1)$, then $\rtau (i,k+1)$ exists and $\rtau (i,k+1)> \rtau (j, k+1)$.
\end{enumerate}
We denote the set of all SRCTs of shape $\alpha$ by $\SRCT (\alpha).$
\end{definition}

Intuitively we can think of the triple rule as saying that if $a< b$ then $a<c$ in the following subarray of cells.$$\tableau{b &c \\ \\&a }$$

The \emph{descent set} of an SRCT $\rtau$ of size $n$, denoted by $\des(\rtau)$, is $$\des(\rtau) = \{ i \suchthat i+1 \mbox{ appears weakly right of } i\}\subseteq [n-1]$$and the corresponding \emph{descent composition} of $\rtau$ is $\comp(\rtau)=\comp(\des(\rtau)).$ Given a composition $\alpha = (\alpha _1 , \ldots , \alpha _k)$, the \emph{canonical tableau of shape} $\alpha$, denoted by $\rtau _\alpha$, is the unique $\SRCT$ of shape $\alpha$ satisfying $\comp({\rtau}_\alpha)= (\alpha _1 , \ldots , \alpha _k)$: Let $\rtau_{\alpha}$ be the $\SRCT$ where the first row is filled with $\alpha _1, \ldots , 2, 1$ and row $i$ for $2\leq i \leq \ell (\alpha)$ is filled with
$$x+\alpha _i, \ldots , x+2, x+1$$where $x=\alpha _1 +\cdots + \alpha _{i-1}$. Then $\rtau_{\alpha}$ is clearly an $\SRCT$ satisfying $\comp(\alpha)=(\alpha_1,\ldots,\alpha_k)$. To see that $\rtau_{\alpha}$ is unique, note that the number of descents in $\rtau_{\alpha}$ is one less than the number of rows in $\tau_{\alpha}$ and so all entries in the first column except $n$ must be all $i$ such that $i\in \des(\rtau_{\alpha})$. This and the fact that $\rtau_{\alpha}$ must be an $\SRCT$ yields uniqueness.

\begin{example}\label{ex:SRCT3432}
Below are two SRCTs of shape $(3,4,3,2)$. The SRCT on the left satisfies $\des (\rtau ) = \{1,2,5,8,9,11 \}$ and hence $\comp (\rtau ) = (1,1,3,3,1,2,1)$. The SRCT on the right is $\rtau _{(3,4,3,2)}$.
$$\rtau = 
\tableau{5&4&2\\8&7&6&3\\11&10&1\\12&9}\qquad 
\rtau _{(3,4,2,3)}=
\tableau{3&2&1\\7&6&5&4\\10&9&8\\12&11}$$
\end{example}

Observe that by the triple rule, the number $1$ will \emph{always} be an entry of a cell that is a removable node.

\subsection{Quasisymmetric functions and quasisymmetric Schur functions}\label{subsec:Qsym} In this subsection we will define two families of functions, one of which will be central to our studies, and one of which will play a major role in a number of proofs. Both families lie in $\bC [[x_1 , x_2 , x_3 , \ldots ]]$, the graded Hopf algebra of formal power series in the commuting variables $x_1 , x_2 , x_3 , \ldots$ graded by total monomial degree, and it is the latter of these families that we introduce first.

\begin{definition}\label{def:Fquasisymmetric}
Let $\alpha\vDash n$ be a composition. Then the
\emph{fundamental quasisymmetric function} $F_{\alpha}$ is defined by $F_\varnothing=1$ and
\[F_{\alpha}=\sum x_{i_1}\cdots x_{i_n}\]
where the sum is over all $n$-tuples $(i_1,\ldots,i_n)$ of positive integers
satisfying
\[\textrm{$1\leq i_1\leq\cdots\leq i_n\:$ and
$\;i_j<i_{j+1}\,$ if $\,j\in\set(\alpha)$.}\] \end{definition}

\begin{example}\label{ex:Falpha}
We have
$F_{(1,2)} = x_1^1x_2^2 +x_1^1x_3^2+x_1^1x_4^2+x_2^1x_3^2+\cdots  + x_1x_2x_3 + x_1x_2x_4 + \cdots$.
\end{example}

We can now define the \emph{Hopf algebra of quasisymmetric functions}, $\Qsym$, which is a graded sub-Hopf algebra of $\bC [[x_1 , x_2 , x_3 , \ldots ]]$ by $$\Qsym = \bigoplus _{n\geq 0} \Qsym ^n$$ where $$\Qsym ^n = \spam\{ F_\alpha\suchthat \alpha \vDash n\} = \spam\{ \qs_\alpha\suchthat \alpha \vDash n\}$$and the latter basis \cite[Proposition 5.5]{HLMvW-1}, which will be our focus of study, is defined as follows.

\begin{definition}\label{def:Squasisymmetric}
Let $\alpha\vDash n$ be a composition. Then the
\emph{quasisymmetric Schur function} $\qs _{\alpha}$ is defined by $\qs _\varnothing=1$ and
\[\qs_{\alpha}=\sum _{\rtau \in \SRCT(\alpha)} F _{\comp (\rtau)}.\]
 \end{definition}
 
\begin{example}\label{ex:Salpha} We compute
$$\qs _{(2,1,3)} = F_{(2,1,3)}+F_{(2,2,2)}+F_{(1,2,1,2)}$$from the following SRCTs, whose descents are marked in red.
$$\tableau{\tcr{2}&1\\\tcr{3}\\6&5&4}
\quad
\tableau{\tcr{2}&1\\\tcr{4}\\6&5&3}
\quad
\tableau{\tcr{3}&\tcr{1}\\\tcr{4}\\6&5&2}$$
\end{example}

The original motivation for naming these functions quasisymmetric Schur functions is the following. If we denote by $\partitionof{\alpha}$ the rearrangement of the parts of $\alpha$ in weakly decreasing order then the renowned Schur functions, which form a basis for the Hopf algebra of symmetric functions $\Sym$, are defined to be $s _{\partitionof{\alpha}}$ where
$$s_{\partitionof{\alpha}} = \sum _{\beta : \partitionof{\beta} = \partitionof{\alpha}} \qs _\beta.$$

\subsection{The symmetric group and the weak Bruhat order}\label{subsec:Sn}
Consider the symmetric group $\sgrp_n$ for any positive integer $n$. It is generated by the adjacent transpositions $s_i=(i,i+1)$ for $1\leq i\leq n-1$ that satisfy the following relations.
\begin{eqnarray*}
s_i^2&=&1 \text{ for } 1\leq i\leq n-1\\s_{i}s_{i+1}s_{i}&=&s_{i+1}s_is_{i+1} \text { for } 1\leq i\leq n-2\\s_{i}s_{j}&=&s_js_{i} \text{ if } \lvert i-j\rvert \ge 2
\end{eqnarray*}
In the first relation above, the $1$ stands for the identity permutation. 
Given a permutation $\sigma \in \sgrp_n$, we can therefore write it as a product of adjacent transpositions. An expression which uses the minimal number of adjacent transpositions is called a \emph{reduced word} for $\sigma$. It is well-known that the number of adjacent transpositions occurring in any reduced word for $\sigma$ is the same. This allows us to define a length function, $l$, on $\sgrp_n$ by letting $l(\sigma)=$ number of adjacent transpositions in any reduced word for $\sigma$.

We will now give another characterization for the length of a permutation.
To do that, we require the definition of the inversion set of a permutation $\sigma\in \sgrp _n$. 
\begin{definition}\label{def:inversions}
Let $\sigma\in\sgrp_n$. Then
\begin{eqnarray*}
Inv(\sigma)=\{(p,q)\suchthat 1\leq p<q\leq n \text{ and } \sigma (p)>\sigma (q)\}.
\end{eqnarray*}
\end{definition}
Classically, we also have $l(\sigma)=|Inv(\sigma)|$ for any $\sigma\in \sgrp_n$.
Next, we discuss an important partial order on $\mathfrak{S}_n$ called the {(left) weak Bruhat order}.
\begin{definition}\label{def:bruhat}
Let $\sigma_1,\sigma_2\in \sgrp_n$. Define the \emph{weak Bruhat order} $\leq_{L}$ on $\sgrp _n$ by $\sigma_1 \leq_{L} \sigma_2$ if and only if there exists a sequence of transpositions $s_{i_1}, \ldots , s_{i_k}\in \sgrp _n$ so that $\sigma_2=s_{i_k}\cdots s_{i_1}\sigma_1$ and $l(s_{i_r}\cdots s_{i_1}\sigma_1)=l(\sigma_1)+r$ for $1\leq r\leq k$.
\end{definition}

An alternative characterization of the weak Bruhat order \cite[Proposition 3.1.3]{bjorner-brenti}, that will also be used later, is 
\begin{eqnarray*}
\sigma_1\leq_{L}\sigma_2 \Longleftrightarrow Inv(\sigma_1)\subseteq Inv(\sigma_2),
\end{eqnarray*}
and it satisfies the following. 

\begin{theorem}\label{the:Bjorner}\cite{bjorner}
Given permutations $\sigma _1, \sigma _2 \in \sgrp _n$ such that $\sigma_1\leq_{L}\sigma_2$, the interval $[\sigma_1,\sigma_2]=\{\sigma\in \mathfrak{S}_n\suchthat \sigma_1\leq_{L}\sigma\leq_{L}\sigma_2\}$ has the structure of a lattice under the weak Bruhat order $ \leq_{L}$ on $\sgrp _n$.
\end{theorem}

\subsection{The \texorpdfstring{$0$-Hecke algebra $\hn$}{0-Hecke algebra} and its representations}\label{subsec:reps}
The $0$-Hecke algebra $H_n(0)$ is the $\mathbb{C}$-algebra generated by the elements $T_1,\ldots,T_{n-1}$ subject to the following relations similar to those of $\sgrp_n$.
\begin{eqnarray*}
T_i^2&=&T_i \text{ for } 1\leq i\leq n-1\\T_{i}T_{i+1}T_{i}&=&T_{i+1}T_iT_{i+1} \text { for } 1\leq i\leq n-2\\T_{i}T_{j}&=&T_jT_{i} \text{ if } \lvert i-j\rvert \ge 2
\end{eqnarray*}

Suppose a permutation $\sigma\in \sgrp_n$ has a reduced word of the form $s_{i_1}\cdots s_{i_p}$. Then we define an element $T_{\sigma}\in H_{n}(0)$ as follows.
\begin{eqnarray*}
T_{\sigma}=T_{i_1}\cdots T_{i_p}
\end{eqnarray*}
The Word Property \cite[Theorem 3.3.1]{bjorner-brenti} of $\mathfrak{S}_n$ together with the defining relations of $H_{n}(0)$ implies that $T_{\sigma}$ is independent of the choice of reduced word. Moreover, the set $\{T_{\sigma}\vert \sigma\in \sgrp_n\}$ is a linear basis for $H_{n}(0)$. Thus, the dimension of $H_{n}(0)$ is $n!$.

Let $\mathcal{R}(H_{n}(0))$ denote the $\mathbb{Z}$-span of the isomorphism classes of finite dimensional representations of $H_{n}(0)$. The isomorphism class corresponding to an $H_{n}(0)$-module $M$ will be denoted by $[M]$.
The \textit{Grothendieck group} $\mathcal{G}_{0}(H_{n}(0))$ is the quotient of $\mathcal{R}(H_{n}(0))$ modulo the relations $[M]=[M']+[M'']$ whenever there exists a short exact sequence $0\to M'\to M\to M''\to 0$. 
The irreducible representations of $H_{n}(0)$ form a free $\mathbb{Z}$-basis for $\mathcal{G}_{0}(H_{n}(0))$. Let
\begin{eqnarray*}
\mathcal{G}=\bigoplus_{n\geq 0}\mathcal{G}_{0}(H_{n}(0)).
\end{eqnarray*}

 According to \cite{norton}, there are $2^{n-1}$ distinct irreducible representations of $H_{n}(0)$. They are naturally indexed by compositions of $n$. Let $\mathbf{F}_{\alpha}$ denote the $1$-dimensional $\mathbb{C}$-vector space corresponding to the composition $\alpha\vDash n$, spanned by a vector $v_{\alpha}$. Let $J\subseteq [n-1]$ be $\set(\alpha)$. Next, we define an action of the generators $T_{i}$ of $H_{n}(0)$ as follows.
 \begin{eqnarray*}
 T_{i}(v_{\alpha})=\left\lbrace\begin{array}{ll}0 & i\in J\\v_{\alpha} & i\notin J\end{array}\right.
 \end{eqnarray*}
 Then $\mathbf{F}_{\alpha}$ is an irreducible $1$-dimensional $H_n(0)$-representation.
 
In \cite{DKLT}, Duchamp, Krob, Leclerc and Thibon define a linear map 
$$\begin{array}{rrcl}
ch:&\mathcal{G}&\longrightarrow& \Qsym \\
&[\mathbf{F}_{\alpha}] &\mapsto &F_{\alpha}.
\end{array} $$
Given an $H_{n}(0)$-module $M$, $ch([M])$ is said to be the \emph{quasisymmetric characteristic} of $M$. It is clear from the definition of the Grothendieck group that every time we have a short exact sequence of $H_{n}(0)$-modules, say $0\to M'\to M\to M''\to 0$, then 
\begin{eqnarray*}
ch([M])&=&ch([M'])+ch([M'']),
\end{eqnarray*}
which will be useful later. More information about the $0$-Hecke algebra of the symmetric group $\hn$ and its representations can be found in \cite{carter,mathas}, and more contemporary results can be found in \cite{huang-1, huang-2}.

\section{A \texorpdfstring{$0$-Hecke}{0-Hecke} action on SRCTs}\label{sec:0Heckeaction}
In order to define an action on SRCTs we first need the concept of attacking.
\begin{definition}\label{def:attack}
Given $\rtau \in \SRCT(\alpha)$ for some composition $\alpha\vDash n$, and a positive integer $i$ such that $1\leq i\leq n-1$, we say that $i$ and $i+1$ are \emph{attacking} if either
\begin{enumerate}
\item $i$ and $i+1$ are in the same column in $\rtau$, or
\item $i$ and $i+1$ are in adjacent columns in $\rtau$, with $i+1$ positioned strictly southeast of $i$.
\end{enumerate} 
\end{definition}

More informally, we say that $i$ and $i+1$ are attacking, if in $\rtau$ they appear as one of
$${\scriptsize\tableau{i\\ \\ \\i\!+\!1} }\quad\mbox{ or }\quad{\scriptsize\tableau{i\!+\!1\\ \\ \\i} }\quad\mbox{ or }\quad {\scriptsize \tableau{i& \\ & \\ & \\ &i\!+\!1}}$$where the first case only occurs in the first column, and the second case only occurs in column $j$ where $j\geq 2$, which follows from the definition of SRCTs.

Let $\rtau\in \SRCT(\alpha)$ where $\alpha \vDash n$. Given a positive integer $i$ satisfying $1\leq i\leq n-1$, let $s_i(\rtau)$ denote the filling obtained by interchanging the positions of entries $i$ and $i+1$ in $\rtau$. 

Given $\rtau \in \SRCT(\alpha)$, define operators $\pi_i$ for $1\leq i\leq n-1$ as follows.
\begin{eqnarray}\label{eq:pi}
\pi_{i}(\rtau)&=& \left\lbrace\begin{array}{ll}\rtau & i\notin \des(\rtau)\\ 0 & i\in \des(\rtau), i \text{ and }  i+1 \text{ attacking}\\ s_{i}(\rtau) & i\in \des(\rtau), i \text{ and } i+1 \text{ non-attacking}\end{array}\right.
\end{eqnarray}
For ease of comprehension, we say if $i\in \des(\rtau)$ with $i$ and $i+1$ attacking then $i$ is an \emph{attacking} descent, whereas if $i\in \des(\rtau)$ with $i$ and $i+1$ non-attacking then $i$ is a \emph{non-attacking} descent.

We will spend the remainder of this section establishing the following theorem.

\begin{theorem}\label{the:0heckerels} The operators $\{ \pi _i \} _{i=1}^{n-1}$ satisfy the same relations as the generators $\{T _i \} _{i=1}^{n-1}$ for the 0-Hecke algebra $\hn$, thus giving an $\hn$-action on SRCTs of size $n$.
\end{theorem}

In order to prove this theorem we need to establish some lemmas, but first we need to recall from \cite[Section 3.5]{VasuMN} the concept of box-adding operators, and introduce some related concepts. Let $\beta \vDash n$. Then the \emph{growth word} associated to $\rtau \in \SRCT (\beta)$ is the word $\tup _{i_1}\cdots\tup _{i_n}$ where $i_j$ is the column in which the entry $j$ appears in $\rtau$. Furthermore, if given a composition $\alpha$, then we define the \emph{box-adding operator} $ \tup _i$ for $i\geq 1$ to be
\begin{eqnarray*}
\tup_{i}(\alpha)=\left\lbrace \begin{array}{ll}(1,\alpha) & \text{if } i=1\\0 & \text{there is no part equal to } i-1 \text{ in }\alpha, i\geq 2\\ \alpha' & \alpha' \text{ obtained by replacing the leftmost } i-1 \text{ in } \alpha \text{ by } i, i\geq 2. \end{array}\right.
\end{eqnarray*}
\begin{example}
Let $\alpha = (3,4,2,2,3)$. Then $\tup_1(\alpha)=(1,3,4,2,2,3)$, $\tup_{3}(\alpha)=(3,4,3,2,3)$, $\tup_4(\alpha)=(4,4,2,2,3)$ and $\tup_{5}(\alpha)=(3,5,2,2,3)$. As there is no part equal to $1$ in $\alpha$, we have that $\tup_2(\alpha)=0$.
\end{example}
We then define the action of a word $\tup _{i_1}\cdots\tup _{i_n}$ on a composition $\alpha$ to be $\tup _{i_1}\cdots\tup _{i_n}(\alpha)$ and can show it satisfies the following two properties.

\begin{lemma}\cite[Lemma 3.10]{VasuMN}\label{lem:fullcommutativity}
Let $i,j$ be positive integers such that $\vert i-j\vert \geq 2$ and $\alpha$ be a composition. Then $\tup_j\tup_i(\alpha)=\tup_i\tup_j(\alpha)$.
\end{lemma}

\begin{lemma}\cite[Lemma 5.13]{VasuMN}\label{lem:commutativityconsecutive}
Let $j\geq 2$ be a positive integer and $\alpha$ be a composition. Suppose that $\alpha$ has parts equal to $j$ and $j-1$ and that the number of parts equal to $j$ which lie to the left of the leftmost instance of a part equal to $j-1$ is $\geq 1$. Then, we have that
\begin{eqnarray*}
\tup_j\tup_{j+1}(\alpha)=\tup_{j+1}\tup_j(\alpha).
\end{eqnarray*}
\end{lemma}

%%% long action-packed proof %%%

\begin{lemma}\label{lem:adjacentnonattacking}
Let $\alpha \vDash n$ and $\rtau \in \SRCT(\alpha)$. Suppose that $i\in \des(\rtau)$ with $i$ in column $k$ and $i+1$ in column $k+1$. Assume further that $i$ and $i+1$ are non-attacking. Then $k\geq 2$.
\end{lemma}

\begin{proof}
Assume, on the contrary, that there exists an $i \in \des(\rtau)$ in the first column with $i+1$ in the second column satisfying the condition that $i$ and $i+1$ are non-attacking. Then the entry in the first column of the row containing $i+1$ is greater than $i$ by definition. But the row containing $i+1$ is strictly above the row containing $i$, and hence we have a contradiction to the fact that the first column of an SRCT increases from top to bottom. The claim now follows.
\end{proof}

\begin{lemma}\label{lem:switchSRCT}
\begin{enumerate}
\item If $\alpha \vDash n$, $\rtau \in  \SRCT(\alpha)$ and $j\in \des(\rtau)$ such that $j$ and $j+1$ are non-attacking, then $s_j(\rtau)\in \SRCT(\alpha)$. 
\item If $\alpha \vDash n$, $\rtau \in  \SRCT(\alpha)$ and $j\notin \des(\rtau)$ such that $j$ is not in the cell to the immediate right of $j+1$, then $s_j(\rtau)\in \SRCT(\alpha)$. 
\end{enumerate}
\end{lemma}
\begin{proof}
For the first part, let the growth word associated to $\rtau$ be $w=\tup_{i_1}\cdots \tup_{i_n}$. Then $w(\varnothing)=\alpha$. We can factorize $w$ as $w=w_1\tup_{i_j}\tup_{i_{j+1}}w_2$. Since $j\in \des(\rtau)$ and $j,j+1$ are non-attacking, we know that $i_j<i_{j+1}$. If $i_{j+1}-i_j \geq 2$, then we know that $\tup_{i_{j+1}}\tup_{i_j}(w_2(\varnothing))=\tup_{i_j}\tup_{i_{j+1}}(w_2(\varnothing))$ by Lemma \ref{lem:fullcommutativity}. Thus, we have that
\begin{eqnarray*}
w_1\tup_{i_{j+1}}\tup_{i_j}(w_2(\varnothing))&=&w_1 \tup_{i_j}\tup_{i_{j+1}}(w_2(\varnothing))\nonumber\\ &=& \alpha.
\end{eqnarray*}
Now, the above equation implies that $w_1\tup_{i_{j+1}}\tup_{i_j}w_2$ corresponds to an SRCT of shape $\alpha$ equal to $s_{j}(\rtau)$.

Now, assume that $i_{j+1}-i_j =1$, that is, $j$ and $j+1$ belong to columns $k$ and $k+1$ respectively. Then the growth word $w$ associated to $\rtau$ factorizes as $w=w_1\tup_{k}\tup_{k+1}w_2$. Since $j$ and $j+1$ are non-attacking by hypothesis, Lemma \ref{lem:adjacentnonattacking} implies that $k\geq 2$ and that, in $\rtau$, the cell containing $j+1$ is strictly northeast of the cell containing $j$. This implies that $w_2(\varnothing)$ contains a part equal to $k$ lying to the left of the leftmost part equal to $k-1$. Thus, Lemma \ref{lem:commutativityconsecutive} implies that $\tup_{k+1}\tup_{k}(w_2(\varnothing))=\tup_{k}\tup_{k+1}(w_2(\varnothing))$. This further implies that 
\begin{eqnarray*}
w_1\tup_{k+1}\tup_{k}(w_2(\varnothing))&=&w_1\tup_{k}\tup_{k+1}(w_2(\varnothing))\nonumber\\ &=&\alpha.
\end{eqnarray*}
Again, from the above equation it follows that $w_1\tup_{{k+1}}\tup_{k}w_2$ corresponds to an SRCT of shape $\alpha$ equal to $s_{j}(\rtau)$. Hence the claim follows.

The proof of the second part is similar.
\end{proof}

\begin{example}\label{ex:heckeaction}
Let $\rtau$ be the SRCT of shape $(3,4,2,3)$ shown below.
$$\tableau{
5&4&2\\9&7&6&3\\10&1\\12&11&8
}$$
Then $\des(\rtau)=\{1,2,5,7,9,10\}$. Thus, for all $1\leq i\leq 11$ such that $i\notin \des(\rtau)$, we have that $\pi_i(\rtau)=\rtau$. Notice further that $2,7,9$ and $10$ are attacking descents. Hence if $i=2,7,9$ or $10$, we have that $\pi_i(\rtau)=0$. Finally, we have that $\pi_1(\rtau)$ and $\pi_5(\rtau)$ are as below.
$$
\pi_1(\rtau)=\tableau{
5&4&1\\9&7&6&3\\10&2\\12&11&8
}
\qquad
\pi_5(\rtau)=\tableau{
6&4&2\\9&7&5&3\\10&1\\12&11&8
}$$
\end{example}
Now we are ready to show that the relations satisfied by the $0$-Hecke algebra $\hn$ are also satisfied by the operators $\{ \pi _i \} _{i=1} ^{n-1}$.

\begin{lemma}\label{lem:pisquared}
For $1\leq i\leq n-1$, we have $\pi_{i}^{2}=\pi_i$.
\end{lemma}
\begin{proof}
Consider an SRCT $\rtau$ of size $n$. If $i\notin \des(\rtau)$, then $\pi_i(\rtau)=\rtau$. Therefore, $\pi_{i}^{2}(\rtau)=\rtau=\pi_i(\rtau)$ in this case.

If $i\in \des(\rtau)$ but $i$ and $i+1$ are attacking, then $\pi_i(\rtau)=0$. Thus, $\pi_{i}^{2}(\rtau)=0=\pi_i(\rtau)$ in this case. Finally, if $i\in \des(\rtau)$ with $i$, $i+1$ non-attacking then $\pi_i(\rtau)=s_i(\rtau)$. Now, in $s_i(\rtau)$, we see that $i\notin \des(s_i(\rtau))$ implying that $\pi_i(s_i(\rtau))=s_i(\rtau)$. Therefore, in this case we have $\pi_i^{2}(\rtau)=s_i(\rtau)=\pi_i(\rtau)$. Thus, in all three cases, we have that $\pi_{i}^{2}=\pi_i$.
\end{proof}

\begin{lemma}\label{lem:pidifferby2}
For $1\leq i,j\leq n-1$ such that $\lvert i-j\rvert \geq 2$, we have $\pi_i\pi_j=\pi_j\pi_i$.
\end{lemma}
\begin{proof}
We will proceed by casework. Consider an SRCT $\rtau$ of size $n$. Suppose first that neither $i$ nor $j$ belong to $\des(\rtau)$. Then $\pi_i(\rtau)=\pi_j(\rtau)=\rtau$. Hence, $\pi_{i}\pi_{j}(\rtau)=\pi_{j}\pi_{i}(\rtau)$ in this case.

Suppose now that $i\in \des(\rtau)$. If $i$ and $i+1$ are attacking, then $\pi_{i}(\rtau)=0$. If $\pi_j(\rtau)=0$, we are done. Hence, assume otherwise. Since $\lvert i-j\rvert\geq 2$, we know that $i\in \des(\pi_j(\rtau))$ and $i$, $i+1$ are still attacking in $\pi_j(\rtau)$. This gives that $\pi_i\pi_j(\rtau)=0$. Thus, in this case as well, we have that $\pi_i\pi_j(\rtau)=\pi_j\pi_i(\rtau)=0$.

Finally, assume that $i\in\des(\rtau)$ and $i$, $i+1$ are non-attacking. Thus, $\pi_i(\rtau)=s_i(\rtau)$. If $j\notin \des(\rtau)$, then the fact that $\lvert i-j\rvert\geq 2$ implies $j\notin \des(s_i(\rtau))$. Thus $\pi_j\pi_i(\rtau)=\pi_j(s_i(\rtau))=s_i(\rtau)=\pi_i\pi_j(\rtau)$. If $j\in \des(\rtau)$ with $j$ and $j+1$ attacking, then the same holds for $s_i(\rtau)$. Thus, $\pi_j(\rtau)=\pi_j(s_i(\rtau))=0$. If $j\in \des(\rtau)$ with $j$ and $j+1$ non-attacking, then $\pi_i\pi_j(\rtau)=s_i(s_j(\rtau))$ and $\pi_j\pi_i(\rtau)=s_j(s_i(\rtau))$. But since $\lvert i-j\rvert \geq 2$, $s_i(s_j(\rtau))$ is the same as $s_j(s_i(\rtau))$.

Thus, we have established that $\pi_i\pi_j=\pi_j\pi_i$.
\end{proof}

\begin{lemma}\label{lem:piiiplus1}
For $1\leq i\leq n-2$, we have $\pi_i\pi_{i+1}\pi_i=\pi_{i+1}\pi_i\pi_{i+1}$.
\end{lemma}
\begin{proof}
Consider an SRCT $\rtau$ of size $n$. We will deal with various cases.\newline
\emph{Case I: $i,i+1\notin \des(\rtau)$}\newline
In this case we have $\pi_i(\rtau)=\pi_{i+1}(\rtau)=\rtau$. Thus, we clearly have that $ \pi_i\pi_{i+1}\pi_i(\rtau)=\pi_{i+1}\pi_i\pi_{i+1}(\rtau)$.

\noindent \emph{Case II: $i\notin \des(\rtau)$, $i+1\in \des(\rtau)$}\newline
In this case we have $\pi_{i}(\rtau)=\rtau$. If $i+1$ and $i+2$ are attacking, then $\pi_{i+1}(\rtau)=0$. This implies that $\pi_i\pi_{i+1}\pi_i(\rtau)=\pi_{i+1}\pi_i\pi_{i+1}(\rtau)=0$. 

Assume therefore that $i+1$ and $i+2$ are non-attacking. Then $\pi_{i+1}(\rtau)=s_{i+1}(\rtau)$. One can see that we have three possibilities in $s_{i+1}(\rtau)$, each of which is dealt with individually in what follows.
\begin{itemize}
\item Consider first the case where $i\notin \des(s_{i+1}(\rtau))$. Then we have $\pi_i\pi_{i+1}(\rtau)=\pi_{i}(s_{i+1}(\rtau))=s_{i+1}(\rtau)$. Thus, we have $\pi_{i+1}\pi_i\pi_{i+1}(\rtau)=s_{i+1}(\rtau)$. Using the fact that $\pi_i(\rtau)=\rtau$ gives us that $\pi_{i+1}\pi_i(\rtau)=s_{i+1}(\rtau)$. Since $i\notin \des(s_{i+1}(\rtau))$ by our assumption, we have that $\pi_i\pi_{i+1}\pi_i(\rtau)=s_{i+1}(\rtau)$, and we are done.
\item Now consider the case where $i\in \des(s_{i+1}(\rtau))$ and $i$, $i+1$ attacking in $s_{i+1}(\rtau)$. Then $\pi_i\pi_{i+1}(\rtau)=\pi_{i}(s_{i+1}(\rtau))=0$ and hence $\pi_{i+1}\pi_i\pi_{i+1}(\rtau)=0$. Now, clearly $\pi_{i+1}\pi_i(\rtau)=\pi_{i+1}(\rtau)=s_{i+1}(\rtau)$. Thus, $\pi_i\pi_{i+1}\pi_i(\rtau)=0$ which concludes this case.
\item Consider finally the case where $i\in \des(s_{i+1}(\rtau))$ with $i$, $i+1$ non-attacking in $s_{i+1}(\rtau)$. Then, $\pi_i\pi_{i+1}(\rtau)=s_is_{i+1}(\rtau)$. Notice that, in $s_is_{i+1}(\rtau)$, $i+1$ is not a descent. Hence $\pi_{i+1}\pi_i\pi_{i+1}(\rtau)=s_is_{i+1}(\rtau)$. This is precisely what $\pi_{i}\pi_{i+1}\pi_i(\rtau)$ is, as $\pi_i(\rtau)$ is just $\rtau$.
\end{itemize}

\noindent\emph{Case III: $i\in \des(\rtau)$, $i+1\notin\des(\rtau)$}\newline
Firstly, notice that we have $\pi_{i+1}(\rtau)=\rtau$. If $i$ and $i+1$ are attacking in $\rtau$, then $\pi_{i}\pi_{i+1}(\rtau)=\pi_{i}(\rtau)=0$ and therefore $\pi_{i+1}\pi_i\pi_{i+1}(\rtau)=0$. Clearly, $\pi_{i}\pi_{i+1}\pi_{i}(\rtau)=0$ as well.

Now, assume that $i$ and $i+1$ are non-attacking. Then $\pi_{i}\pi_{i+1}(\rtau)=s_i(\rtau)$. Again, we have the following three possibilities in $s_i(\rtau)$.
\begin{itemize}
\item Assume that $i+1\notin \des(s_i(\rtau))$. Then $\pi_{i+1}(s_i(\rtau))=s_i(\rtau)$. Thus, in particular, we have that $\pi_{i+1}\pi_{i}\pi_{i+1}(\rtau)=s_i(\rtau)$. 
On the other hand we have that $\pi_i\pi_{i+1}\pi_i(\rtau)=\pi_{i}\pi_{i+1}(s_i(\rtau))=\pi_i(s_i(\rtau))=s_i(\rtau)$.
\item Assume now that $i+1\in \des(s_i(\rtau))$ with $i+1$, $i+2$ attacking in $s_i(\rtau)$. Then $\pi_{i+1}\pi_i(\rtau)=\pi_{i+1}(s_i(\rtau))=0$. This gives that $\pi_i\pi_{i+1}\pi_i(\rtau)=0$. Now, we have that $\pi_i\pi_{i+1}(\rtau)=\pi_i(\rtau)=s_i(\rtau)$. Thus, $\pi_{i+1}\pi_i\pi_{i+1}(\rtau)=0$ too.
\item The third case is the one where $i+1\in \des(s_i(\rtau))$ and $i+1$, $i+2$ are non-attacking in $s_i(\rtau)$. Then $\pi_{i+1}\pi_i(\rtau)=s_{i+1}s_i(\rtau)$. Notice that $i\notin \des(s_{i+1}s_i(\rtau))$. Therefore, $\pi_i\pi_{i+1}\pi_{i}(\rtau)=s_{i+1}s_i(\rtau)$. We have that $\pi_{i+1}\pi_{i}\pi_{i+1}(\rtau)=\pi_{i+1}\pi_i(\rtau)=s_{i+1}s_i(\rtau)$, and this settles the last case here.
\end{itemize}
\noindent\emph{Case IV: $i\in \des(\rtau)$, $i+1\in \des(\rtau)$} \newline
Assume first the scenario where $i$ and $i+1$ are attacking in $\rtau$. In this case, $\pi_i(\rtau)=0$ implying $\pi_{i}\pi_{i+1}\pi_i(\rtau)=0$. We have the following two situations arising.
\begin{itemize}
\item If $i+1$ and $i+2$ are also attacking in $\rtau$, then it is immediate that $\pi_{i+1}\pi_i\pi_{i+1}(\rtau)=0$. 
\item Assume now that $i+1$ and $i+2$ are non-attacking in $\rtau$. Then we claim that $i$ and $i+1$ are non-attacking in $\pi_{i+1}(\rtau)=s_{i+1}(\rtau)$ and that $i\in \des(s_{i+1}(\rtau))$. The fact that $i\in \des(s_{i+1}(\rtau))$ is immediate. If $i+1$ and $i+2$ occupy columns whose indices differ by at least 2 in $\rtau$, then it is clear that $i$ and $i+1$ will be non-attacking in $s_{i+1}(\rtau)$. Assume therefore that $i+1$ and $i+2$ occupy adjacent columns in $\rtau$. If $i$ and $i+1$ are also in adjacent columns of $\rtau$, then they occur in columns that are distance exactly 2 apart in $s_{i+1}(\rtau)$. Thus, this also gives that $i$ and $i+1$ are non-attacking in $s_{i+1}(\rtau)$. Now assume that $i$ and $i+1$ are both in column $k$ in $\rtau$. Since $i+1$ and $i+2$ are non-attacking and in adjacent columns in $\rtau$, we know by Lemma \ref{lem:adjacentnonattacking}, that $k\geq 2$. Thus, in $\rtau$, $i+1$ occupies a cell that is strictly above the cell occupied by $i$. Since, in $\rtau$, $i+2$ occupies a cell that is strictly northeast of the cell containing $i+1$, we get that $i$ and $i+1$ will be non-attacking in $s_{i+1}(\rtau)$.

Thus $\pi_{i}\pi_{i+1}(\rtau)=s_is_{i+1}(\rtau)$. Now notice that $i+1\in \des(s_is_{i+1}(\rtau))$ and that $i+1$ and $i+2$ are attacking in $s_is_{i+1}(\rtau)$. This gives us that $\pi_{i+1}\pi_{i}\pi_{i+1}(\rtau)=0$.
\end{itemize}
Assume now the second scenario that $i$ and $i+1$ are non-attacking in $\rtau$. Then $\pi_{i}(\rtau)=s_i(\rtau)$. Since $i+1\in \des(\rtau)$, we are guaranteed that $i+1\in \des(s_i(\rtau))$. Moreover, $i+1$ and $i+2$ are non-attacking in $s_i(\rtau)$. To see this, we will use very similar analysis as we did earlier. 

If $i+1$ and $i+2$ are non-attacking in $\rtau$, then we get that $i+1$ and $i+2$ occupy columns whose indices differ by at least 2 in $s_i(\rtau)$. Hence they are clearly non-attacking in this case. If $i+1$ and $i+2$ are attacking and in adjacent columns, we still reach the conclusion that $i+1$ and $i+2$ occupy columns whose indices differ by at least 2 in $s_i(\rtau)$. Thus, in this case too, we get that $i+1$ and $i+2$ are non-attacking in $s_i(\rtau)$. Assume finally that $i+1$ and $i+2$ are in the same column, say $k$, of $\rtau$. If $i$ and $i+1$ are in columns whose indices differ by at least 2 in $\rtau$, we know that $i+1$ and $i+2$ will be non-attacking in $s_i(\rtau)$. Hence, assume that $i$ and $i+1$ are in adjacent columns. Hence $i$ is in column $k-1$ and $i+1$ occupies a cell strictly northeast of the cell containing $i$. Lemma \ref{lem:adjacentnonattacking} implies that $k-1\geq 2$, which is equivalent to $k\geq 3$. Thus, we know that, in $\rtau$, $i+2$ occupies a cell in column $k$ that is strictly north of the cell occupied by $i+1$. Hence, in $s_i(\rtau)$, $i+1$ and $i+2$ are indeed non-attacking.

Thus $\pi_{i+1}\pi_{i}(\rtau)=s_{i+1}s_{i}(\rtau)$. Now we will consider two cases.
\begin{itemize}
\item If $i+1$ and $i+2$ are attacking in $\rtau$, then $i\in \des(s_{i+1}s_i(\rtau))$ with $i$ and $i+1$ attacking. Thus $\pi_i\pi_{i+1}\pi_i(\rtau)=0$ in this case. However, so is $\pi_{i+1}\pi_i\pi_{i+1}(\rtau)$ as $\pi_{i+1}(\rtau)=0$ already.

\item Assume that $i+1$ and $i+2$ are non-attacking in $\rtau$. This implies that $i\in \des(s_{i+1}s_i(\rtau))$ with $i$ and $i+1$ non-attacking too. Thus, $\pi_i\pi_{i+1}\pi_i(\rtau)=s_is_{i+1}s_{i}(\rtau)$. 

Now $\pi_{i+1}(\rtau)=s_{i+1}(\rtau)$ and it is clear that $i\in \des(s_{i+1}(\rtau))$ with $i$, $i+1$ non-attacking. Thus, $\pi_i\pi_{i+1}(\rtau)=s_is_{i+1}(\rtau)$. Since $i$ and $i+1$ are non-attacking in $\rtau$ and $i\in \des(\rtau)$, we are guaranteed that $i+1\in \des(s_is_{i+1}(\rtau))$ with $i+1$, $i+2$ non-attacking. Thus $\pi_{i+1}\pi_{i}\pi_{i+1}(\rtau)=s_{i+1}s_is_{i+1}(\rtau)$. One can check that $s_is_{i+1}s_i(\rtau)=s_{i+1}s_is_{i+1}(\rtau)$ by noticing that in both cases the positions of $i$ and $i+2$ in $\rtau$ have been interchanged. This completes the proof. 
\end{itemize}
\end{proof}

The proof of Theorem~\ref{the:0heckerels} now follows immediately from Lemmas~\ref{lem:pisquared}, \ref{lem:pidifferby2} and \ref{lem:piiiplus1}.

\section{The partial order \texorpdfstring{$\po$}{} on SRCTs}\label{sec:partialorder}
Since the operators $\{\pi_i\}_{i=1}^{n-1}$, which we will now term \emph{flips} for convenience, satisfy the same relations as the $0$-Hecke algebra $\hn$ by Theorem~\ref{the:0heckerels}, we can associate a well-defined linear operator $\pi_{\sigma}$ with any permutation $\sigma\in \sgrp_n$, as in Subsection \ref{subsec:reps}. We will use these operators to define a new partial order on SRCTs of the same shape. However, before that we need some definitions and results, and to recall that we say $\sigma \in \sgrp _n$ is written in \emph{one-line notation} if we write it as a word $\sigma(1)\sigma(2)\cdots \sigma(n)$.

\begin{definition}\label{def:column_word}Let $\alpha \vDash n$ whose largest part is $\alpha _{max}$, and $\rtau \in \SRCT(\alpha)$, whose entries in column $i$ for $1\leq i \leq \alpha _{max}$ read from top to bottom are some word $w^i$. Then we define the \emph{column word} of $\rtau$, denoted by $\col_{\rtau}$, to be the word
$$w^1 \  w^2\cdots w^{\alpha _{max}}$$
and identify it with the natural permutation of $\sgrp _n$ written in one-line notation.
\end{definition}

\begin{example}\label{ex:colword} Let $\rtau$ be the SRCT from Example~\ref{ex:heckeaction}. Then 
$$\col _{\rtau}=5\ 9\ 10\ 12\ 4\ 7\ 1\ 11\ 2\ 6\ 8\ 3,$$which can also be viewed as a permutation in $\sgrp _{12}$ in one-line notation.
\end{example}

Let $\alpha\vDash n$ and $\rtau_1\in \SRCT(\alpha)$ be such that $i\in \des(\rtau_1)$. Suppose further that $i$ is a non-attacking descent in $\rtau_1$. Let $\pi_i(\rtau_1)=\rtau_2$. Then observe that $s_i\col_{\rtau_1}=\col_{\rtau_2}$ as permutations, where $s_i=(i,i+1)$ is the adjacent transposition in $\sgrp_n$ that interchanges $i$ and $i+1$. 

\begin{lemma}\label{lem:poflips}
Let $\alpha \vDash n$ and $\rtau _1, \rtau _2 \in \SRCT (\alpha)$ such that $\pi _{j_1} \cdots \pi _{j_r}(\rtau _1)=\rtau _2$ for $1\leq j_1, \ldots , j_r \leq n-1$. Let $s_{i_1}\cdots s_{i_p}$ be a reduced word for $\col_{\rtau_2}(\col_{\rtau_1})^{-1}$. Then $\pi_{i_1}\cdots\pi_{i_p}(\rtau_1)=\rtau_2$.
\end{lemma}

\begin{proof}
Note that, given a permutation $\sigma \in \sgrp _n$, by Theorem~\ref{the:0heckerels} and the rules for multiplying generators of $\hn$ we have that
\begin{eqnarray*}
\pi_i\pi_{\sigma}=\left\lbrace\begin{array}{ll}\pi_{s_i\sigma} & l(s_i\sigma)>l(\sigma)\\ \pi_{\sigma} & l(s_i\sigma)<l(\sigma)\end{array}\right.
\end{eqnarray*}is satisfied.
Thus, if $\pi_{j_1}\cdots \pi_{j_r}(\rtau_1)=\rtau_2$ then we can find a permutation $\tilde{\sigma}\in \sgrp_n$ such that $\pi_{j_1}\cdots \pi_{j_r}=\pi_{\tilde{\sigma}}$ in $H_{n}(0)$. Moreover, since $\pi _{\tilde{\sigma}}(\rtau _1)=\rtau _2$ we have that $\tilde{\sigma}\col _{\rtau _1}=\col _{\rtau _2}$. Thus,
 any reduced word $s_{i_1}\cdots s_{i_p}$ for $\col_{\rtau_2}(\col_{\rtau_1})^{-1}$ is a reduced word for $\tilde{\sigma}$ and $\pi_{i_1}\cdots\pi_{i_p}(\rtau_1)=\pi_{\tilde{\sigma}}(\rtau_1)=\rtau_2$.
\end{proof}

\begin{lemma}\label{lem:precimpliesbruhat}Let $\alpha\vDash n$ and $\rtau_1\in \SRCT(\alpha)$ such that $i\in \des(\rtau_1)$. Suppose further that $i$ is a non-attacking descent in $\rtau_1$. If $\pi_i(\rtau_1)=\rtau_2$, then $l(\col_{\rtau_2})=l(\col_{\rtau_1})+1$ and $\col_{\rtau_1} \leq _{L} \col_{\rtau_2}$.
\end{lemma}

\begin{proof} If $i\in \des (\rtau _1)$ and is non-attacking, then by definition $i$ occurs to the left of $i+1$ in $\col _{\rtau _1}$. Thus $s_i \col _{\rtau _1}= \col _{\rtau _2}$ contains one more inversion than $\col _{\rtau _1}$ and $Inv (\col_{\rtau_1}) \subseteq Inv (\col_{\rtau_2})$, and the result follows.
\end{proof}

We note down an immediate corollary concerning the weak Bruhat order from the lemma above.

\begin{corollary}\label{cor:precimpliesweakleftbruhat} If we obtain an SRCT $\rtau _2$ starting from an SRCT $\rtau _1$ through a sequence of flips, then $\col _{\rtau _1} \leq _L \col _{\rtau _2}$.
\end{corollary}

From the above lemma it also follows that if we can obtain an SRCT $\rtau _2$ starting from an SRCT $\rtau _1$ through a sequence of flips, where $\rtau _1 \neq \rtau _2$, then there does not exist a sequence of flips to obtain $\rtau _1$ starting from $\rtau _2$. Now we are in a position to define a new partial order on SRCTs of the same shape.

\begin{definition}\label{def:partialorder}Let $\alpha \vDash n$ and $\rtau _1, \rtau _2 \in \SRCT (\alpha)$. Define a partial order $\po $ on $\SRCT(\alpha)$ by $\rtau _1 \po \rtau _2$ if and only if there exists a permutation $\sigma \in \sgrp _n$ such that $\pi _\sigma (\rtau _1) = \rtau _2$.
\end{definition}

The reflexivity and transitivity are immediate from the definition. The anti-symmetricity follows from the observation preceding the above definition. Thus we indeed have a \emph{partial order}.

\section{\texorpdfstring{$0$-Hecke}{0-Hecke} modules from SRCTs and quasisymmetric Schur functions}\label{sec:0Hecke}
We will now use the partial order from the previous section to define an $\hn$-module indexed by a composition $\alpha \vDash n$, whose quasisymmetric characteristic is the quasisymmetric Schur function $\qs _\alpha$. More precisely, given a composition $\alpha \vDash n$ extend the partial order $\po$ on $\SRCT (\alpha)$ to an arbitrary total order on $\SRCT (\alpha)$, denoted by $\po ^t$. Let the elements of $\SRCT (\alpha)$ under $\po ^t$ be $$\{\rtau _1 \prec _\alpha ^t  \cdots \prec _\alpha ^t \rtau _m  \}.$$ Now let $\mathcal{V}_{\rtau_i}$ be the $\mathbb{C}$-linear span 
 $$\mathcal{V}_{\rtau_i} = \spam \{ \rtau _j \suchthat \rtau _i \po ^t \rtau _j \}\quad \text{ for } 1\leq i \leq m$$
and observe that the definition of $\po ^t$ implies that $\pi_{\sigma}\mathcal{V}_{\rtau_i}\subseteq \mathcal{V}_{\rtau_i}$ for any $\sigma\in \sgrp_n$. This observation combined with the fact that the operators $\{\pi_i\}_{i=1}^{n-1}$ satisfy the same relations as $\hn$ by Theorem~\ref{the:0heckerels} gives the following result.
\begin{lemma}\label{lem:hnmodule}
$\mathcal{V}_{\rtau_i}$ is an $\hn$-module.
\end{lemma}

Given the above construction, define $\mathcal{V}_{\rtau_{m+1}}$ to be the trivial $0$ module, and consider the following filtration of $\hn$-modules.
$$
\mathcal{V}_{\rtau_{m+1}}\subset \mathcal{V}_{\rtau_m} \subset \cdots \subset \mathcal{V}_{\rtau_{2}}\subset \mathcal{V}_{\rtau_1} 
$$
Then, the quotient modules $\mathcal{V}_{\rtau_{i-1}}/\mathcal{V}_{\rtau_{i}}$ for $2\leq i\leq m+1$ are $1$-dimensional $\hn$-modules spanned by $\rtau_{i-1}$. Furthermore, they are irreducible modules. We can identify which $\hn$-module they are by looking at the action of $\pi_{j}$ on $\mathcal{V}_{\rtau_{i-1}}/\mathcal{V}_{\rtau_{i}}$ for $1\leq j\leq n-1$. We have
\begin{eqnarray*}
\pi_{j}(\rtau_{i-1})&=& \left \lbrace \begin{array}{ll}0 & j\in \des(\rtau_{i-1})\\\rtau_{i-1} & \text{otherwise.}\end{array}\right.
\end{eqnarray*}
Thus, as an $\hn$-representation, $\mathcal{V}_{\rtau_{i-1}}/\mathcal{V}_{\rtau_{i}}$ is isomorphic to the irreducible representation $\mathbf{F}_{\beta}$ where $\beta$ is the composition corresponding to the descent set $\des(\rtau_{i-1})$. Thus, by Subsection~\ref{subsec:reps}
\begin{eqnarray*}
ch([\mathcal{V}_{\rtau_{i-1}}/\mathcal{V}_{\rtau_{i}}])=F_{\beta}.
\end{eqnarray*}
From this it follows that 
\begin{eqnarray*}
ch([\mathcal{V}_{\rtau_1}])&=&\displaystyle\sum_{i=2}^{m+1}ch([\mathcal{V}_{\rtau_{i-1}}/\mathcal{V}_{\rtau_{i}}])\nonumber\\ 
&=& \displaystyle \sum_{i=2}^{m+1}F_{\comp(\rtau_{i-1})}\nonumber\\ 
&=& \displaystyle \sum_{\rtau \in \SRCT(\alpha)}F_{\comp(\rtau)}\nonumber\\ 
&=& \qs_{\alpha}.
\end{eqnarray*}
Consequently, we have established the following.

\begin{theorem}\label{the:bigone}
Let $\alpha \vDash n$ and $\rtau _1 \in \SRCT (\alpha)$ be the minimal element under the total order $\po ^t$. Then $\mathcal{V}_{\rtau_1}$ is an $\hn$-module whose quasisymmetric characteristic is the quasisymmetric Schur function $\qs_{\alpha}$.\end{theorem}

Henceforth, if given such a $\rtau _1 \in \SRCT (\alpha)$ where $\alpha \vDash n$, we will denote the $\hn$-module $\mathcal{V}_{\rtau_1}$ by $\smodule_{\alpha}$, that is, $\smodule_{\alpha}:=\mathcal{V}_{\rtau_1}$ and
\begin{equation}\label{eq:qsmodule}ch([\smodule_{\alpha}])= \qs _\alpha.\end{equation}

\section{Source and sink tableaux, and the weak Bruhat order}\label{sec:sourcesink}In this section we define an equivalence relation that will be essential when determining indecomposability of $\hn$-modules in the next section. Additionally, we study two types of tableaux, called source and sink tableaux, which will allow us to realise the equivalence classes as subintervals of the weak Bruhat order on $\sgrp _n$. To facilitate this we first need the concept of standardization.

\begin{definition}\label{def:stadardization} Given a word $w=w_1\cdots w_n$ such that $w_1, \ldots, w_n$ are distinct positive integers, we say that the \emph{standardization} of $w$ is the permutation $\sigma \in \sgrp _n$ such that $\sigma(i) < \sigma (j)$ if and only if $w_i <w_j$ for $1\leq i,j \leq n$.
\end{definition}

\begin{definition}\label{def:st} Let $\alpha \vDash n$ whose largest part is $\alpha _{max}$, and $\rtau \in \SRCT(\alpha)$, whose entries in column $i$ for $1\leq i \leq \alpha _{max}$ read from top to bottom are some word $w^i$. Then we say the \emph{standardized $i$-th column word} of $\rtau$, denoted by $\st_i(\rtau)$ is the standardization of $w^i$.

Furthermore, we define the \emph{standardized column word} of $\rtau$, denoted by $\st (\rtau)$ to be the word
$$\st (\rtau) = \st_1(\rtau) \  \st_2(\rtau) \cdots \st_{\alpha _{max}}(\rtau).$$
\end{definition}

\begin{example}\label{ex:st} If $\rtau$ is the SRCT below from Example~\ref{ex:SRCT3432}, then $\st (\rtau ) = 1234 \ 1243\ 231\ 1.$
$$\tableau{5&4&2\\8&7&6&3\\11&10&1\\12&9}$$
\end{example}

Note that since the number of entries in each column of an SRCT $\rtau$ of shape $\alpha$ weakly decreases, in order to recover $\st _1$ from $\st(\rtau)$ we need only search for the first instance when an integer is repeated in order to determine the beginning of $\st _2$, and then iterate this to recover every $\st _j$ for $1\leq j \leq  {\alpha _{max}}$.

With this definition of standardized column word in hand we will now define a natural but crucial equivalence relation, whose reflexivity, symmetricity and transitivity are immediate.

\begin{definition}\label{def:equiv} Let $\alpha \vDash n$ and $\rtau_1, \rtau _2 \in \SRCT (\alpha)$. Define an equivalence relation $\equiva$ on $\SRCT(\alpha)$ by 
$$\rtau _1 \equiva \rtau _2 \text{ if and only if } \st(\rtau _1)=\st(\rtau _2).$$
\end{definition}

\begin{example}\label{ex:equiv}Note the following SRCTs are equivalent under $\sim _{(3,2,4)}$, as for every $\rtau$ below we have that $\st (\rtau ) = 123\ 123\ 12\ 1.$$$
\tableau{
 {3}&2&1\\{5}&4\\9&8&7&6
}
\hspace{3mm}
\tableau{
 {3}&2&1\\ {6}&4 \\9&8&7&5
}
\hspace{3mm}
\tableau{
 {3}&2&1\\ {7}& {4}\\9&8&6&5
}
\hspace{3mm}
\tableau{
 {3}&2&1\\ {7}& {5}\\9&8&6&4
}
\hspace{3mm}
\tableau{
3&2&1\\6&5\\9&8&7&4
}
$$
\medskip
$$
\tableau{
4&2&1\\6&5\\9&8&7&3
}
\hspace{3mm}
\tableau{
 {4}& {2}&1\\ {7}& {5}\\9&8&6&3
}
\hspace{3mm}
\tableau{
4&3&1\\6&5\\9&8&7&2
}
\hspace{3mm}
\tableau{
 {4}&3& {1}\\ {7}& {5}\\9&8&6&2
}
$$
\end{example}

Suppose given a composition $\alpha \vDash n$ the equivalence classes with respect to $\equiva$ are $E_1,E_2,\cdots, E_k$. Let $\smodule_{\alpha,E_i}$ denote the $\mathbb{C}$-linear span of all SRCTs in $E_i$ for $i=1, \ldots, k$. Then we get the following isomorphism of vector spaces
\begin{equation}\label{eq:Ssum}
\smodule_{\alpha}\cong\bigoplus_{i=1}^{k}\smodule_{\alpha,E_i}
\end{equation}
since the equivalence classes are disjoint and their union is $\SRCT(\alpha)$. This isomorphism is actually an isomorphism of $\hn$-modules, as the following lemma implies immediately. 

\begin{lemma}\label{lem:Emodule} Let $E_j$ for $j= 1, \ldots , k$ be the equivalence classes under $\equiva$ for $\alpha \vDash n$. Then
for all $i$ such that $1\leq i\leq n-1$, we have that $$\pi_i(\smodule_{\alpha,E_j}) \subseteq \smodule_{\alpha,E_j}$$for any $1\leq j\leq k$.
\end{lemma}

\begin{proof}Consider $\rtau_1\in E_j$, and suppose further that $\pi_i(\rtau_1)=\rtau_2$ for some positive integer $i$. In particular we are assuming that $i$ is a non-attacking descent in $\rtau_1$. Therefore, all the entries strictly greater than $i$ in the same column as $i$ in $\rtau_1$ are actually all strictly greater than $i+1$. Also, all entries strictly lesser than $i+1$ in the same column as $i+1$ in $\rtau_1$ are actually all strictly lesser than $i$. Thus, once the positions of $i$ and $i+1$ in $\rtau_1$ are interchanged to obtain $\rtau_2$, we have that $\st (\rtau_2)=\st (\rtau_1)$. Thus, $\rtau_2$ belongs to $E_j$ as well.
\end{proof}

Now fix an equivalence class $E$. Given $\rtau \in E$, a removable node of $\alpha$ is called a \textit{distinguished removable node} of $\rtau$ if it contains the smallest entry in that column. Since all SRCTs in $E$ have the same standardized column word, it is clear that the distinguished removable nodes only depend on the equivalence class $E$, and not on the individual tableaux therein. Thus, we can define a set of positive integers, $DRN(\alpha,E)$ as follows.
\begin{align*}
DRN(\alpha,E)=\{r \suchthat & \text{ there exists a distinguished removable node in column } r\\&\text{ in any tableau in } E \}.
\end{align*}
Note that the set $DRN(\alpha,E)$ is clearly a non-empty set, as any cell that contains $1$ in any tableau in $E$ is a distinguished removable node.
Furthermore, there is at most one distinguished removable node in any given column of any tableau in $E$.

Next, we discuss two important classes of SRCTs that will form special representatives of each equivalence class.

\begin{definition}\label{def:source}
Let $\alpha\vDash n$. An SRCT $\rtau$ of shape $\alpha$ is said to be a \emph{source} tableau if it satisfies the condition that for every $i\notin \des(\rtau)$ where $i\neq n$, we have that $i+1$ lies to the immediate left of $i$.
\end{definition}
\begin{definition}\label{def:sink}
Let $\alpha\vDash n$. An SRCT $\rtau$ of shape $\alpha$ is said to be a \emph{sink} tableau if it satisfies the condition that for every $i\in \des(\rtau)$, we have that $i$ and $i+1$ are attacking.  
\end{definition}

\begin{example}\label{ex:sourcesink}
Let $\alpha=(4,3,2,3)$. We list an example each of a source tableau (left) and sink tableau (right) of shape $\alpha$.
\begin{eqnarray*}
\tableau{
7&6&5&4\\8&3&2\\9&1\\12&11&10
}\hspace{5mm}
\tableau{
8&6&3&1\\9&5&2\\10&4\\12&11&7
}
\end{eqnarray*}Observe that they both have standardized column word $1234\ 3214\ 213\ 1$ and so are in the same equivalence class $E$. They both have a distinguished removable node in columns $2,3,4$ and hence
$$DRN((4,3,2,3), E) = \{2,3,4\}.$$
\end{example}

The motivation for these names is immediate from the following two lemmas. 

\begin{lemma}\label{lem:source}
Let $\alpha\vDash n$ and $\rtau_1\in \SRCT(\alpha)$. Then the following are equivalent.
\begin{enumerate}
\item $\rtau_1$ is a source tableau.
\item There does not exist an SRCT $\rtau_2\neq \rtau_1$, such that $\pi_i(\rtau_2)=\rtau_1$ for some positive integer $i$.
\end{enumerate}
\end{lemma}

\begin{proof}
Assume first that $\rtau_1$ is a source tableau. Suppose, for the sake of contradiction, the existence of an SRCT $\rtau_2\neq \rtau_1$ such that $\pi_i(\rtau_2)=\rtau_1$. Then $i\in \des(\rtau_2)$, but $i\notin \des(\rtau_1)$. Since $i$ is clearly not $n$, we know that $i+1$ lies to the immediate left of $i$ in $\rtau_1$. But, if we obtained $\rtau_1$ from $\rtau_2$ by interchanging the positions of $i$ and $i+1$, we reach the conclusion that $i$ lies to the immediate left of $i+1$ in $\rtau_2$. But this clearly contradicts the fact that $\rtau_2$ is an SRCT.  

Now assume that there does not exist an SRCT $\rtau_2\neq \rtau_1$ such that $\pi_i(\rtau_2)=\rtau_1$ for some positive integer $i$. In view of the second part of Lemma~\ref{lem:switchSRCT}, it is immediate that $\rtau_1$ has the property that if $i$ is such that $i+1$ lies strictly left of $i$, that is $i\notin \des(\rtau _1)$, then $i+1$ actually lies to the immediate left of $i$. But this is the same as saying that $\rtau_1$ is a source tableau.
\end{proof}

\begin{lemma}\label{lem:sink}
Let $\alpha\vDash n$ and $\rtau_1\in \SRCT(\alpha)$. Then the following are equivalent.
\begin{enumerate}
\item $\rtau_1$ is a sink tableau. 
\item There does not exist an SRCT $\rtau_2\neq \rtau_1$, such that $\pi_i(\rtau_1)=\rtau_2$ for some positive integer $i$.
\end{enumerate}
\end{lemma}
\begin{proof}
Assume first that $\rtau_1$ is a sink tableau.
If $i\notin \des(\rtau_1)$, we know that $\pi_i(\rtau_1)=\rtau_1$. If, on the other hand, $i\in \des(\rtau_1)$, we know that $\pi_i(\rtau_1)=0$ as $i$ and $i+1$ are attacking in $\rtau_1$. Thus, there does not exist an SRCT $\rtau_2\neq \rtau_1$ such that $\pi_i(\rtau_1)=\rtau_2$ for some positive integer $i$.

Now assume that there does not exist an SRCT $\rtau_2\neq \rtau_1$ such that $\pi_i(\rtau_1)=\rtau_2$ for some positive integer $i$. This implies that $\pi_i(\rtau_1)$ either equals $\rtau_1$ or it equals $0$ for any $1\leq i\leq n-1$. This is the same as saying that any positive integer $1\leq i\leq n-1$ satisfies the condition that either $i\notin \des(\rtau_1)$ or $i$ is an attacking descent in $\rtau_1$. But this is the same as saying that $\rtau_1$ is a sink tableau.
\end{proof}

These two lemmas now allow us to determine some equivalence class structure.

\begin{lemma}\label{lem:sourcesinkexistence} Let $\alpha\vDash n$ and $E$ be an equivalence class under the equivalence relation $\equiva$. Then $E$ contains at least one source tableau and at least one sink tableau of shape $\alpha$.
\end{lemma}

\begin{proof} We know by Lemma~\ref{lem:precimpliesbruhat} that if $\pi _i (\rtau _1)=\rtau _2$ then $l(\col_{\rtau_2})=l(\col_{\rtau_1})+1$. The result now follows by Lemmas~\ref{lem:source} and \ref{lem:sink}, and the fact that $|\SRCT (\alpha)|$ is finite, and hence so is the number of tableaux in $E$.
\end{proof}

Let $\rtau$ be a source tableau. Let $\rtau_{-1}$ be the SRCT obtained by first removing the cell containing $1$ from $\rtau$ and then subtracting $1$ from the remaining entries. Then it is easy to see that $\rtau_{-1}$ is also a source tableau.

Every equivalence class under the equivalence relation contains at least one source tableau by Lemma~\ref{lem:sourcesinkexistence}. The next proposition helps us prove that each equivalence class contains exactly one source tableau.

\begin{proposition}\label{prop:position_of_1_source}
Let $\rtau$ be an SRCT of shape $\alpha\vDash n$ belonging to an equivalence class $E$. Let $\rtau_{source}$ be a source tableau in $E$. Let $m$ be the smallest element in $DRN(\alpha,E)$. Then the distinguished removable node in column $m$ of $\rtau_{source}$ contains the entry $1$.
\end{proposition}
\begin{proof}
Note first that any column that $1$ belongs to in any tableau in $E$ will always be an element of $DRN(\alpha,E)$. Let the entry in the distinguished removable node in column $m$ of $\rtau$ be $i\geq 2$. We will exhibit the existence of a positive integer $j$ satisfying $1\leq j\leq i-1$, and an SRCT $\rtau_1 \neq \rtau$ such that $\pi_j(\rtau_1)=\rtau$.

Let $1$ belong to column $m'>m$ in $\rtau$. Let
\begin{eqnarray*}
X=\{p\suchthat 1\leq p\leq i-1 \text{ and } p \text{ belongs to a column } r>m \}.
\end{eqnarray*}
Since $i\geq 2$ and hence $1\in X$, the set $X$ is clearly non-empty. We will consider two cases.

Suppose first that $i-1 \in X$, and assume moreover that it belongs to column $m+1$. Since $i$ is the entry of a distinguished removable node and thus at the end of a row, we get that $i-1$ can not belong to the same row as $i$. Furthermore, the triple rule implies that $i-1$ belongs to a row strictly north of the row containing $i$. We can deduce from the second part of Lemma~\ref{lem:switchSRCT} that the filling $\rtau_1$ obtained by interchanging the positions of $i$ and $i-1$ in $\rtau$ is still an SRCT. Then, clearly we have that $\pi_{i-1}(\rtau_1)=\rtau$.

Suppose now that $i-1 \in X$, and that it belongs to column $j$ where $j \geq m+2$. Then it is immediate from Lemma~\ref{lem:switchSRCT} that the filling $\rtau_1$ obtained by interchanging the positions of $i$ and $i-1$ in $\rtau$ is still an SRCT. As before, we have that $\pi_{i-1}(\rtau_1)=\rtau$.

If $i-1$ does not belong to $X$, then let $x_{max}$ denote the maximal element of $X$. Since $x_{max}+1 \leq i-1$ and $x_{max}+1 \notin X$, we get that $x_{max}+1$ belongs to column $q$ where $q\leq m$. However, since $i$ is the entry of a distinguished removable node in column $m$, it is the smallest entry in its column and hence the stronger inequality $q<m$ holds. Thus, $x_{max}$ and $x_{max}+1$ belong to columns whose indices differ by at least 2. Hence by the second part of Lemma~\ref{lem:switchSRCT} the filling $\rtau_1$ obtained by interchanging the positions of $x_{max}$ and $x_{max}+1$ is still an SRCT. Again, we have that $\pi_{x_{max}}(\rtau_1)=\rtau$.
%If $i-1$ does not belong to $X$, let $x_{max}$ denote the maximal element of $X$. We know that $x_{max}+1$ belongs to column $q$ where $q<m$, since $i$ is the entry of a distinguished removable node so $q\neq m$. Hence by the second part of Lemma~\ref{lem:switchSRCT} the filling $\rtau_1$ obtained by interchanging the positions of $x_{max}$ and $x_{max}+1$ is still an SRCT. Again, we have that $\pi_{x_{max}}(\rtau_1)=\rtau$.
\end{proof}

Now, let $\rtau$ be a sink tableau. Then it is easy to see that $\rtau_{-1}$ defined above is also a sink tableau. As with source tableaux, every equivalence class under the equivalence relation contains at least one sink tableau, by using Lemma \ref{lem:sourcesinkexistence}. Likewise, the purpose of the next proposition is to help us prove that each equivalence class contains exactly one sink tableau.

\begin{proposition}\label{prop:position_of_1_sink}
Let $\rtau$ be an SRCT of shape $\alpha\vDash n$ belonging to an equivalence class $E$. Let $\rtau_{sink}$ be a sink tableau in $E$. Let $M$ be the largest element in $DRN(\alpha,E)$. Then the distinguished removable node in column $M$ of $\rtau_{sink}$ contains the entry $1$.
\end{proposition}

\begin{proof}
As noted in the previous proof, any column that $1$ belongs to in any tableau in $E$ will always be an element of $DRN(\alpha,E)$. Let the entry in the distinguished removable node in column $M$ of $\rtau$ be $i\geq 2$. We will exhibit the existence of a positive integer $j$ satisfying $1\leq j\leq i-1$, and an SRCT $\rtau_1\neq \rtau$ such that $\pi_j(\rtau)=\rtau_1$.

Let $1$ belong to column $M'<M$ in $\rtau$. Let
\begin{eqnarray*}
X=\{p\suchthat 1\leq p\leq i-1 \text{ and } p \text{ belongs to a column } r<M \}.
\end{eqnarray*}
Since $i\geq 2$ and $1\in X$, the set $X$ is clearly non-empty. We will consider two cases.

If $i-1$ belongs to $X$, then we claim that $i-1$ is a non-attacking descent in $\rtau$. Let $c_{i-1}$ be the column to which $i-1$ belongs in $\rtau$. If $M-c_{i-1}\geq 2$, then it is immediate that $i-1$ is a non-attacking descent. Thus, we are left to consider the case where $c_{i-1}=M-1$. Suppose that $i-1$ is actually an attacking descent in $\rtau$. Hence, we know that the row to which $i-1$ belongs is strictly above the row to which $i$ belongs. Since, $i$ occupies a removable node in column $M$, we know that the row to which $i-1$ belongs has $\geq M$ cells. But then the entry to the immediate right of $i-1$ in $\rtau$ is $\leq i-2$. This contradicts the fact that the smallest entry in column $M$ is in the cell that contains $i$. Hence $i-1$ belongs to a row that is strictly below the row containing $i$, and $i-1$ is a non-attacking descent in $\rtau$. Then, clearly we have that $\pi_{i-1}(\rtau)=\rtau_1$.

If $i-1$ does not belong to $X$, let $x_{max}$ denote the maximal element of $X$. Since $x_{max}+1 \leq i-1$ and $x_{max}+1 \notin X$, we get that $x_{max}+1$ belongs to column $q$ where $q\geq M$. However, since $i$ is the entry of a distinguished removable node in column $M$, it is the smallest entry in its column and hence the stronger inequality $q>M$ holds. Thus, $x_{max}$ and $x_{max}+1$ belong to columns whose indices differ by at least 2, and this implies that $x_{max}$ is a non-attacking descent in $\rtau$. Hence by the first part of Lemma \ref{lem:switchSRCT} the filling $\rtau_1$ obtained by interchanging the positions of $x_{max}$ and $x_{max}+1$ is still an SRCT. Again, we have that $\pi_{x_{max}}(\rtau)=\rtau_1$.
\end{proof}

These two propositions show that the position of $1$ in a source or sink tableau in an equivalence class $E$ is uniquely decided. Thus, we have the following important corollary.

\begin{corollary}\label{cor:uniquesourceandsink} Let $\alpha \vDash n$. Then in any equivalence class $E$ under $\equiva$, there exists a unique source tableau and a unique sink tableau. Thus, 
\begin{align*}
&\text{ the number of source tableaux}\\
= & \text{ the number of sink tableaux}\\
=& \text{ the number of equivalence classes under the equivalence relation $\equiva$.}\end{align*}
\end{corollary}

With this uniqueness in mind, given a composition $\alpha \vDash n$ and an equivalence class $E$ under $\equiva$, we will now denote the unique source tableau in $E$ by $\rtau_{0,E}$, and denote the unique sink tableau in $E$ by $\rtau_{1,E}$. 

As a corollary to the above corollary, we obtain the following that describes the structure of the $\hn$-modules from \eqref{eq:Ssum}.

\begin{corollary}\label{cor:equivcyclic}
Let $\alpha \vDash n$ and $E$ be an equivalence class under the equivalence relation $\equiva$. Then $\smodule_{\alpha,E}$ is a cyclic $\hn$-module with generator $\rtau_{0,E}$. In particular, every tableau in $E$ can be obtained by applying a sequence of flips to $\rtau_{0,E}$.
\end{corollary}

Thus, we also have the following for $\rtau_1,\rtau_2\in \SRCT(\alpha)$.
\begin{eqnarray}\label{eq:sameorbit}
\rtau_1\equiva\rtau_2 \Longleftrightarrow \rtau_1\text{ and } \rtau_2\text{ lie in the orbit of the same source tableau}.
\end{eqnarray}

Next, we will give an example of the source tableau and the sink tableau for the equivalence class of the tableau $\rtau$ considered in Example \ref{ex:heckeaction}.

\begin{example}\label{ex:DRN}
The distinguished removable nodes for the equivalence class, $E$, of $\rtau$ are marked in red in the diagram below.
\begin{eqnarray*}\tableau{
5&4&\tcr{2}\\9&7&6&3\\10&\tcr{1}\\12&11&8
}
\end{eqnarray*}
We know that the source tableau in $E$ will have $1$ in the distinguished removable node in column $2$ by Proposition~\ref{prop:position_of_1_source} while the sink tableau in $E$ will have a $1$ in the distinguished removable node in column $3$ by Proposition~\ref{prop:position_of_1_sink}. This gives that the source tableau and the sink tableau have $1$ in the following positions below. Repeating this analysis on the remaining unfilled cells, we get the following unique source tableau and unique sink tableau for the equivalence class $E$.
\begin{eqnarray*}
\rtau_{0,E}=\tableau{
4&3&2\\8&7&6&5\\9&1\\12&11&10
}
\hspace{10mm}
\rtau_{1,E}=\tableau{
6&4&1\\9&7&5&2\\10&3\\12&11&8
}
\end{eqnarray*}
\end{example}

Now that we know that two SRCTs $\rtau_1$ and $\rtau_2$ satisfying $\st(\rtau_1)\neq \st(\rtau_2)$ are incomparable under the partial order $\po$, we will focus our attention on the partial order $\po$ restricted to all tableaux belonging to the same equivalence class. 

\begin{theorem}\label{the:bruhatordersubinterval}
Let $\alpha \vDash n$. Let $E$ be a fixed equivalence class of $\SRCT(\alpha)$ under the equivalence relation $\equiva$.Then the poset $(E,\po)$ has the structure of a graded lattice and is isomorphic to the subinterval $[\col_{\rtau_{0,E}},\col_{\rtau_{1,E}}]$ under the weak Bruhat order $\leq_{L}$ on $\sgrp_n$.
\end{theorem}
\begin{proof}
Consider the map $f:E \to \sgrp_n$ given by mapping an SRCT $\rtau$ to $\col_{\rtau}$. This map is clearly injective. Consider a maximal chain $$\rtau_{0,E} \po \pi_{i_1}(\rtau_{0,E})\po \cdots \po \pi_{i_r}\cdots \pi_{i_1}(\rtau_{0,E})=\rtau_{1,E}.$$
We claim that $f$ induces a saturated chain in $(\mathfrak{S}_n,\leq_{L})$ shown below
$$\col_{\rtau_{0,E}}=f(\rtau_{0,E} )\leq_{L} f(\pi_{i_1}(\rtau_{0,E}))\leq_{L} \cdots \leq_{L} f(\pi_{i_r}\cdots \pi_{i_1}(\rtau_{0,E}))=\col_{\rtau_{1,E}}.$$
This follows from Lemma \ref{lem:precimpliesbruhat}.

In a similar vein, suppose one has a saturated chain in $\sgrp _n$ as shown below
$$\col_{\rtau_{0,E}}\leq_{L} s_{j_1}\col_{\rtau_{0,E}}\leq _{L}\cdots \leq_{L} s_{j_r}\cdots s_{j_1}(\col_{\rtau_{0,E}})=\col_{\rtau_{1,E}}.$$
We have that $s_{j_r}\cdots s_{j_1}$ is a reduced word for $\col_{\rtau_{1,E}}\col_{\rtau_{0,E}}^{-1}$. Hence $\pi_{j_r}\cdots \pi_{j_1}(\rtau_{0,E})=\rtau_{1,E}$ by Lemma~\ref{lem:poflips}. Thus, we get a chain in $E$ as shown below
$$\rtau_{0,E} \po \pi_{j_1}(\rtau_{0,E})\po \cdots \po \pi_{j_r}\cdots \pi_{j_1}(\rtau_{0,E})=\rtau_{1,E}.$$
This chain is a maximal chain as $s_{j_r}\cdots s_{j_1}$ is a reduced word expression for $\col_{\rtau_{1,E}}\col_{\rtau_{0,E}}^{-1}$.
Thus, we see that the map taking an element $\sigma\in \mathfrak{S}_n$ satisfying $\col_{\rtau_{0,E}}\leq_{L}\sigma\leq_{L}\col_{\rtau_{1,E}}$ to $\pi_{\sigma\col_{\rtau_{0,E}}^{-1}}(\rtau_{0,E})$ is actually the inverse for $f$. This establishes that $f$ is an isomorphism from $E$ to its image in $\mathfrak{S}_n$ equal to the interval $[\col_{\rtau_{0,E}},\col_{\rtau_{1,E}}]$. Since it maps maximal chains to maximal chains, we get the fact that $f$ is a poset isomorphism.
\end{proof}

\section{The classification of tableau-cyclic and indecomposable modules}\label{sec:cyclic_and_indecomposable_modules}
Elaborating on the $0$-Hecke modules discovered in Section~\ref{sec:0Hecke}, in this section we will classify those $\hn$-modules $\smodule _\alpha$ for $\alpha \vDash n$ that are cyclically generated by a single SRCT, termed \emph{tableau-cyclic}, and use these to classify those that are indecomposable.

The equivalence relation from Section~\ref{sec:sourcesink} will be key to proving our classification, and the following compositions will be key to describing our classification:
We call a composition $\alpha = (\alpha _1, \ldots , \alpha _{\ell(\alpha)})$ \emph{simple} if and only if it satisfies the following condition.
\begin{itemize}
\item If $\alpha _i \geq \alpha _j \geq 2$ and $1\leq i < j \leq \ell(\alpha)$ then there exists an integer $k$ satisfying $i<k<j$ such that $\alpha _k = \alpha _j -1$.
\end{itemize}
Otherwise we call a composition \emph{complex}.

\begin{example}\label{ex:simple} The compositions $(2,5,6)$and $(4,1,2,3,4)$ are simple, whereas $(2,2)$, $(3,1,3)$ and $(5,1,2,4)$ are complex.
\end{example}

The removable nodes of a simple composition are particularly easy to describe.

\begin{lemma}\label{lem:11.1} Let $\alpha = (\alpha _1, \ldots , \alpha _{\ell(\alpha)})$ be a simple composition. Then the part $\alpha _j$ for $1\leq j \leq \ell(\alpha)$ has a removable node if and only if
\begin{enumerate}
\item $j=1$ or
\item for all $1\leq i < j$ we have $\alpha _i \leq \alpha _j -2$.
\end{enumerate}
\end{lemma}

\begin{proof} The result is clear for $j=1$. Hence assume that $j\geq 2$. Then if $\alpha _j$ has a removable node, by definition $\alpha _j \geq 2$. Notice further that if $\alpha _j$ has a removable node, then all parts $\alpha _i$ such that $1\leq i <j$ satisfying $\alpha _i <\alpha _j$ actually satisfy $\alpha _i \leq \alpha _j -2$. Thus we need only establish that there does not exist a part $\alpha _i$ satisfying $1\leq i <j $ and $\alpha _i \geq \alpha _j$.

Suppose on the contrary, that there exists such an $\alpha _i$. Then since $\alpha $ is simple, we are guaranteed the existence of a part $\alpha _k$ such that $i<k<j$ with $\alpha _k=\alpha _j -1$. However, this implies that $\alpha _j$ cannot have a removable node unless $j=1$, which contradicts our assumption on $j$. The reverse direction follows from the definitions and we are done.
\end{proof}

We now prove an elementary yet vital property of simple compositions, and then work towards classifying tableau-cyclic modules.

\begin{lemma}\label{lem:11.2} Suppose $\alpha = (\alpha _1, \ldots , \alpha _{\ell(\alpha)})$ is a simple composition and $\alpha _j$ is a part that has a removable node. Then the composition $(\alpha _1, \ldots ,\alpha _j - 1, \ldots , \alpha _{\ell(\alpha)})$ is also a simple composition.
\end{lemma}

\begin{proof}
The result is clear for $j=1$. Hence assume that $j\geq 2$. Let $\beta= (\beta_1,\ldots,\beta_{\ell(\beta)})$ denote the composition $(\alpha _1, \ldots ,\alpha _j - 1, \ldots , \alpha _{\ell(\alpha)})$. Thus, we clearly have $\ell(\beta)=\ell(\alpha)$ and $\beta_k=\alpha_k$ for all $1\leq k\leq \ell(\alpha)$ except for $k=j$.

From Lemma \ref{lem:11.1}, we have $\alpha_i \leq \alpha_j-2$ for all $1\leq i<j$. This implies that $\beta_i \leq \beta_j-1$ for all $1\leq i<j$. 
Now notice that, since $\alpha$ is a simple composition, it follows that if $\beta_i \geq \beta_k$ where $1\leq i<k<j$, then there exists a positive integer $m$ such that $i<m<k$ and $\beta_m=\beta_k-1$. Similarly, it also follows that if $\beta_i \geq \beta_k$ where $j\leq i<k\leq \ell(\beta)$, then there exists a positive integer $m$ such that $i<m<k$ and $\beta_m=\beta_k-1$.

To establish that $\beta$ is simple, we only need to show that if $\beta_i \geq \beta_k$ where $1\leq i<j<k\leq \ell(\beta)$, then there exists a positive integer $m$ such that $i<m<k$ and $\beta_m=\beta_k-1$. Since $\beta_i=\alpha_i$ and $\beta_k=\alpha_k$, and $\alpha$ is a simple composition, we know that there exists a positive integer $m$ such that $i<m<k$ and $\alpha_m=\alpha_k-1$. Furthermore, as $\alpha_i \leq \alpha_j-2$, we are guaranteed that $m$ does not equal $j$. This together with the fact that $\alpha_m=\beta_m$ finishes the proof that $\beta$ is a simple composition.
\end{proof}

\begin{lemma}\label{lem:11.3} Let $\alpha \vDash n$ be a simple composition. If $\rtau \in \SRCT (\alpha)$ then in every column of $\rtau$ the entries increase from top to bottom.
\end{lemma}

\begin{proof} We will proceed by induction on the size of $\alpha$. The base case $\alpha = (1)$ is clear. Assume the claim holds for all simple compositions of size $n-1$.

In $\rtau$, let $1$ belong to a cell in column $k$. By observing that any cell containing $1$ must be a removable node and Lemma~\ref{lem:11.1} we know that there are no cells strictly north of the cell containing $1$. Hence $1$ lies strictly north of every entry in column $k$.

Now let $\rtau _{-1}$ be the SRCT obtained by removing the cell containing $1$, and subtracting $1$ from the remaining entries. Observe that by Lemma~\ref{lem:11.2} $\rtau _{-1} \in \SRCT(\alpha ')$ where $\alpha '$ is a simple composition of size $n-1$. Hence the induction hypothesis implies that the entries in every column of $\rtau _{-1}$ increase from top to bottom. Since $1$ lies strictly north of every entry in column $k$ in $\rtau$, we can conclude that the entries in every column of $\rtau$ increase from top to bottom.
\end{proof}

\begin{lemma}\label{lem:11.4}Let $\alpha \vDash n$ be a complex composition. Then there exists a $\rtau \in \SRCT (\alpha)$ such that in at least one column of $\rtau$ the entries do not increase from top to bottom.
\end{lemma}

\begin{proof} Since $\alpha = (\alpha _1, \ldots , \alpha _{\ell(\alpha)})$ is a complex composition, we know that there exist $1\leq i < j < \ell(\alpha)$ such that $\alpha _i \geq \alpha _j \geq 2$ and there is no $k$ satisfying $i<k<j$ such that $\alpha _k = \alpha _j -1$. Consider the following compositions $\alpha _{upper} = (\alpha _1, \ldots , \alpha _{i-1})$ and $\alpha _{lower} =(\alpha _i , \ldots , \alpha _{\ell(\alpha)})$. Let $\rtau _{upper}\in \SRCT (\alpha _{upper} )$ and note that the $(j-(i-1))$-th part of $\alpha _{lower}$ has a removable node inherited from $\alpha _j$. Choose $\rtau _{lower} \in \SRCT (\alpha _{lower})$ to be an SRCT with a 1 in this removable node. Furthermore let $\rtau _{lower} + |\alpha _{upper}|$ be the array obtained by adding $ |\alpha _{upper}|$ to every entry of $\rtau _{lower}$.

Now consider $\rtau \in \SRCT(\alpha)$ whose first $i-1$ rows are identical to $\rtau _{upper}$ and whose remaining rows are identical to $\rtau _{lower} + |\alpha _{upper}|$. Observe that the entries in column $j$ of $\rtau$ do not strictly increase from top to bottom, establishing the claim.
\end{proof}

We are now in a position to classify those $\hn$-modules that are tableau-cyclic.
 
 \begin{theorem}\label{the:tableau_cyclic} $\smodule _\alpha$ is a tableau-cyclic $\hn$-module if and only if $\alpha$ is a simple composition of $n$.
 \end{theorem}
 
 \begin{proof} Let $\alpha$ be a simple composition. Then by Lemma~\ref{lem:11.3} we know that there is only one equivalence class under $\equiva$. Thus since 
 $$\smodule_{\alpha}\cong\bigoplus_{E}\smodule_{\alpha,E}$$where the sum is over the equivalence classes $E$ of $\equiva$, by Corollary~\ref{cor:equivcyclic} it follows that $\smodule _\alpha$ is a tableau-cyclic $\hn$-module if $\alpha$ is a simple composition of $n$.
 
 Conversely, let $\alpha$ be a complex composition. Then since the entries in each column of $\rtau _\alpha$ increase from top to bottom, by Lemma~\ref{lem:11.4} we are guaranteed the existence of at least two equivalence classes under the equivalence relation $\equiva$. Thus since 
 $$\smodule_{\alpha}\cong\bigoplus_{E}\smodule_{\alpha,E}$$where the sum is over the equivalence classes $E$ of $\equiva$, by Corollary~\ref{cor:equivcyclic} it follows that $\smodule _\alpha$ is not a tableau-cyclic $\hn$-module if $\alpha$ is a complex composition of $n$.
\end{proof}

Observe that if $\smodule _\alpha$ is not tableau-cyclic then it may still be cyclic. For example, $\smodule _{(2,2)}$ is not tableau-cyclic by Theorem~\ref{the:tableau_cyclic} since $(2,2)$ is a complex composition. However, it is cyclically generated by the following generator.
$$\tableau{2&1\\4&3}\ + \ \tableau{3&2\\4&1}$$We will now work towards classifying those $\hn$-modules that are indecomposable, which is a result that will sound familiar when we state it.

\begin{lemma}\label{lem:previous} Let $\alpha \vDash n$ and $\rtau \in \SRCT (\alpha)$ such that $\rtau \neq \rtau _\alpha$. Then there exists a positive integer $i \not\in \set (\alpha)$ such that $\pi _i (\rtau ) \neq \rtau$ but $\pi _i (\rtau _\alpha) = \rtau _\alpha$.
\end{lemma}

\begin{proof} Since $\rtau _\alpha$ is the unique SRCT of shape $\alpha$ satisfying $\comp (\rtau_\alpha ) = \alpha$ it follows that $\des (\rtau ) \neq \des (\rtau _\alpha)$. Furthermore, $|\des (\rtau )| \geq |\des (\rtau _\alpha)|$ since every element in the first column except the largest contributes to the descent set of an SRCT. Hence we are guaranteed the existence of a positive integer $i$ such that $i \in \des (\rtau)$ but $i \not\in \des (\rtau _\alpha)$, and moreover, $\pi _i (\rtau ) \neq \rtau$ but $\pi _i (\rtau _\alpha) = \rtau _\alpha$.
\end{proof}

Now we are ready to identify those $\hn$-modules that are indecomposable.

\begin{theorem}\label{the:indecomposable} $\smodule _\alpha$ is an indecomposable $\hn$-module if and only if $\alpha$ is a simple composition of $n$.
\end{theorem}

\begin{proof}
Let $\alpha$ be a simple composition. We will show that every idempotent module morphism from $\smodule_{\alpha}$ to itself is either the zero map or the identity map. This in turn implies that $\smodule_{\alpha}$ is indecomposable \cite[Proposition 3.1]{jacobson}.
 
Let $f$ be an idempotent module endomorphism of $\smodule_{\alpha}$. Suppose that 
\begin{eqnarray}
f(\rtau_{\alpha})=\displaystyle\sum_{\rtau\in \SRCT(\alpha)}a_{\rtau}\rtau.
\end{eqnarray}
By Lemma~\ref{lem:previous} we know that given $\rtau _1\in \SRCT(\alpha)$ such that $\rtau _1\neq \rtau_{\alpha}$ there exists a positive integer $i\notin \set(\alpha)$ such that $\pi_i(\rtau _1)\neq \rtau _1$
 but $\pi_i(\rtau_{\alpha})=\rtau_{\alpha}$. The fact that $f$ is a module morphism implies that $f(\pi_i(\rtau_{\alpha}))=\pi_i(f(\rtau_{\alpha}))$. However, $\pi_i(\rtau_{\alpha})=\rtau_{\alpha}$, so we get that 
\begin{eqnarray*}
f(\rtau_{\alpha})&=&f(\pi_i(\rtau_{\alpha}))\nonumber\\&=&\pi_i(f(\rtau_{\alpha}))\nonumber\\ &=& \displaystyle\sum_{{\rtau}\in \SRCT(\alpha)}a_{{\rtau}}\pi_i({\rtau}).
\end{eqnarray*}
Now, note that the coefficient of $\rtau _1\neq \rtau _\alpha$ in the sum in the last line above is $0$. This is because there does not exist a $\rtau_{2}\in \SRCT(\alpha)$ such that $\pi_i(\rtau_{2})=\rtau _1$ (else if $\pi_i(\rtau_{2})=\rtau _1$ then $\pi_i(\rtau_{2})=\pi_i\pi_i(\rtau_{2})=\pi_i(\rtau _1)\neq\rtau _1$, a contradiction).

Thus, we get that $f(\rtau_{\alpha})=a_{\rtau_{\alpha}}\rtau_{\alpha}$. Again, using the fact that $f$ is idempotent, we get that $a_{\rtau_{\alpha}}=0$ or $a_{\rtau_{\alpha}}=1$. Since $\smodule_{\alpha}$ is cyclically generated by $\rtau_{\alpha}$ by Theorem~\ref{the:tableau_cyclic}, we get that $f$ is either the zero map or the identity map. Hence, $\smodule_{\alpha}$ is indecomposable.

Conversely, if $\alpha$ is a complex composition, then since the entries in each column of $\rtau _\alpha$ increase from top to bottom, by Lemma~\ref{lem:11.4} we are guaranteed the existence of at least two equivalence classes under the equivalence relation $\equiva$. Thus since 
  $$\smodule_{\alpha}\cong\bigoplus_{E}\smodule_{\alpha,E}$$where the sum is over the equivalence classes $E$ of $\equiva$, it follows that $\smodule _\alpha$ is not indecomposable.
\end{proof}

\section{The canonical basis and enumeration of truncated shifted reverse tableaux}\label{sec:canonical_basis} Recall that given any composition $\alpha$, we denote its canonical tableau, which is the unique SRCT of shape $\alpha$ and descent composition $\alpha$, by $\rtau _\alpha$. In this section we discover a new basis for $\Qsym$ arising from the orbit of each canonical tableau and connect it to the enumeration of truncated shifted reverse tableaux studied in \cite{adin-king-roichman, panova}. For ease of notation we denote the orbit of $\rtau _\alpha$ by $E_\alpha$, and since the entries in every column of $\rtau _\alpha$ increase from top to bottom, we have by definition that
\begin{align*}
E_\alpha = \{ \rtau \suchthat & \rtau \in \SRCT (\alpha) \text{ and entries in each column of }\rtau \\
&\text{ increase from top to bottom } \}.
\end{align*}

Repeating our argument from Section~\ref{sec:0Hecke}, but now extending the partial order $\po$ on $E_\alpha$ to a total order on $E_\alpha$, we obtain the expansion of the quasisymmetric characteristic of the $\hn$-module $\smodule _{\alpha, E_\alpha}$ in terms of fundamental quasisymmetric functions as follows.
\begin{equation}\label{eq:canonicalchar}
ch([\smodule _{\alpha , E_{\alpha}}]) = \sum _{\rtau \in E_\alpha}F _{\comp(\rtau)}
\end{equation}
Denoting the above function by $C_\alpha$ leads us to the following definition.

\begin{definition}\label{def:canonical} Let $\alpha \vDash n$. Then we define the \emph{canonical quasisymmetric function} $C_\alpha$ to be
\begin{equation}\label{eq:canonical}
C_\alpha = \sum_{\rtau} F _{\comp(\rtau)}
\end{equation}
where the sum is over all SRCT $\rtau$ of shape $\alpha$ whose entries in each column increase from top to bottom.
\end{definition}

In turn, this leads naturally to a new basis for $\Qsym$. However, before we reveal this new basis, we need to recall two total orders on compositions remembering that we denote by $\partitionof{\alpha}$ the rearrangement of the parts of $\alpha$ in weakly decreasing order. First, given distinct compositions $\alpha,\beta \vDash n$, we say that $\alpha$ is \emph{lexicographically greater} than $\beta$, denoted by $\alpha >_{lex} \beta$, if $\alpha =(\alpha_1,\alpha_2,\ldots)$ and $\beta=(\beta_1,\beta_2,\ldots)$ satisfy the condition that for the smallest positive integer $i$ such that $\alpha_i\neq \beta_i$, one has $\alpha_i>\beta_i$. Second, we say that $\alpha \blacktriangleright \beta$ if $\partitionof{\alpha} >_{lex} \partitionof{\beta}$ or $\partitionof{\alpha}=\partitionof{\beta}$ and $\alpha >_{lex} \beta$.

\begin{proposition}\label{prop:canonicalbasis}
The set $\{{C}_{\alpha}\suchthat \alpha\vDash n\}$ forms a $\mathbb{Z}$-basis for $\Qsym ^n$.
\end{proposition}

\begin{proof}
We know from \cite[Proposition 5.5]{HLMvW-1} that the transition matrix expressing the quasisymmetric Schur functions in terms of the basis of fundamental quasisymmetric functions is upper unitriangular when the indexing compositions are ordered using $\blacktriangleright$. Thus, we immediately get that the transition matrix $M$ expressing $\{{C}_{\alpha}\suchthat \alpha\vDash n\}$ in terms of $\{ {F}_{\alpha}\suchthat \alpha\vDash n\}$ is also upper triangular with the indexing compositions ordered using $\blacktriangleright$. Now, since $\rtau_{\alpha}\in E_{\alpha}$, by the definition of $C_\alpha$ above we have that $ {F}_{\alpha}$ appears with coefficient $1$ in this expansion of ${C}_{\alpha}$. Thus, the matrix $M$ is also upper unitriangular and this implies the claim.
\end{proof}

\begin{proposition}\label{prop:canonicalindecomposable}
The set of $\hn$-modules $\{\smodule_{\alpha,E_{\alpha}} \suchthat \alpha\vDash n\}$ is a set of pairwise non-isomorphic indecomposable modules.
\end{proposition}
\begin{proof}
The fact that $\{\smodule_{\alpha,E_{\alpha}} \suchthat \alpha\vDash n\}$ consists of pairwise non-isomorphic $\hn$-modules follows immediately from Proposition~\ref{prop:canonicalbasis}. The proof of the fact that each of these modules is indecomposable is exactly the same as in the proof of Theorem \ref{the:indecomposable}.
\end{proof}

\begin{example}\label{ex:canonical}
Let $\alpha=(3,2,4)$. Then $E_\alpha$ consists of the following SRCTs.
$$
\tableau{
 \tcr{3}&2&1\\\tcr{5}&4\\9&8&7&6
}
\hspace{3mm}
\tableau{
 \tcr{3}&2&1\\ \tcr{6}&\tcr{4} \\9&8&7&5
}
\hspace{3mm}
\tableau{
 \tcr{3}&2&1\\ \tcr{7}& \tcr{4}\\9&8&6&5
}
\hspace{3mm}
\tableau{
 \tcr{3}&2&1\\ \tcr{7}& \tcr{5}\\9&8&6&4
}
\hspace{3mm}
\tableau{
\tcr{3}&2&1\\\tcr{6}&5\\9&8&7&4
}
$$
\medskip
$$
\tableau{
\tcr{4}&\tcr{2}&1\\\tcr{6}&5\\9&8&7&3
}
\hspace{3mm}
\tableau{
 \tcr{4}& \tcr{2}&1\\ \tcr{7}& \tcr{5}\\9&8&6&3
}
\hspace{3mm}
\tableau{
\tcr{4}&3&\tcr{1}\\\tcr{6}&5\\9&8&7&2
}
\hspace{3mm}
\tableau{
 \tcr{{}4}&3& \tcr{1}\\ \tcr{7}& \tcr{5}\\9&8&6&2
}
$$

In the list above, the top-leftmost tableau is the source tableau $\rtau _{0, E_\alpha} = \rtau _\alpha$, while the bottom-rightmost tableau is the sink tableau $\rtau _{1, E_\alpha}$, and the descents are marked in red. Thus, we calculate the canonical quasisymmetric function ${C} _{(3,2,4)}$ to be
\begin{eqnarray*}
{C} _{(3,2,4)}&=&F_{(3,2,4)}+F_{(3,1,2,3)}+F_{(3,1,3,2)}+F_{(3,2,2,2)}+F_{(3,3,3)}\\&&+F_{(2,2,2,3)}+F_{(2,2,1,2,2)}+F_{(1,3,2,3)}+F_{(1,3,1,2,2)}.
\end{eqnarray*}
\end{example}

 \subsection{Dimensions of certain \texorpdfstring{$\smodule_{\alpha,E_{\alpha}}$}{classes} and truncated shifted reverse tableaux}
For most of this subsection, we will be interested in compositions $\alpha=(\alpha_1,\ldots,\alpha_k)$ satisfying the additional constraint that $\alpha_1 < \cdots < \alpha_k$. We will refer to such compositions as \textit{strict reverse partitions}.

\begin{definition}\label{def:shifted_reverse_diagram}
Given a strict reverse partition $\alpha=(\alpha_1,\ldots,\alpha_k)$, the \emph{shifted reverse diagram} of $\alpha$ contains $\alpha_i$ cells in the $i$-th row from the top with the additional condition that, for $2\leq i\leq k$, row $i$ starts one cell to the left of where row $i-1$ starts.
\end{definition}

\begin{definition}\label{def:shifted_reverse_tableau}
Given a strict reverse partition $\alpha=(\alpha_1,\ldots,\alpha_k)\vDash n$, a \emph{shifted reverse tableau} of shape $\alpha$ is a filling of the cells of the shifted reverse diagram of $\alpha$ with distinct positive integers from $1$ to $n$ so that
\begin{enumerate}
\item the entries decrease from left to right in every row, and
\item the entries increase from top to bottom in every column.
\end{enumerate}
\end{definition}

\begin{example}\label{ex:shiftedrt}
Let $\alpha=(2,4,5)$. A shifted reverse tableau of shape $\alpha$ is given below.
$$\tableau{
  &   & 2 &1\\  & 8 & 6 & 5 & 3\\ 11& 10 & 9 & 7 & 4 
}$$
\end{example}

Next, we will discuss {truncated shifted reverse tableaux}. Let $\alpha=(\alpha_1,\ldots,\alpha_k)$ and $\beta=(\beta_1,\ldots,\beta_s)$ be strict reverse partitions such that $s\leq k$ and $\beta_{s-i+1}\leq \alpha_{k-i+1}$ for $1\leq i\leq s$. We define the \textit{truncated shifted reverse diagram} of shape $\alpha\backslash \beta$ as the array of cells where row $i$ starts one cell to the left of where row $i-1$ starts and contains $\alpha_i$ cells if $1\leq i\leq k-s$, and $\alpha_i-\beta_{i+s-k}$ cells if $k-s +1 \leq i \leq k$. Here again, the rows are considered from top to bottom.
This given, we define a \emph{truncated shifted reverse tableau} of \emph{shape} $\alpha \backslash
\beta$ to be a filling satisfying the same conditions as a shifted reverse tableau.

\begin{example}\label{ex:tsrt}
Given below is a truncated shifted reverse tableau of shape $(2,4,5)\backslash(1,2)$.
$$\tableau{
  &   & 2 &1\\  & 5 & 4 & 3 & \bullet\\ 8& 7 & 6 &\bullet &\bullet 
}$$
\end{example}

Given positive integers $a\geq b$, let $\delta _{[a,b]}$ denote the strict reverse partition $(b,b+1,\ldots, a-1,a)$. If $b=1$, we will denote $\delta_{[a,b]}$ by just $\delta_{a}$. 

Now we have all the notation that we need to give a count for the dimension of $\smodule_{\alpha,E_{\alpha}}$ for certain $\alpha$.

\begin{theorem}\label{the:dimsmodule}
Let $\alpha=(\alpha_1,\ldots,\alpha_k)$ be a strict reverse partition. Then the number of SRCTs that belong to $E_{\alpha}$, that is $\dim \smodule _{\alpha, E_\alpha }$, is equal to the number of shifted reverse tableaux of shape $\alpha$.
\end{theorem}

\begin{proof}
 Let $\rtau\in E_{\alpha}$. Consider a cell $(i,j)$ that belongs to the reverse composition diagram of $\alpha$ such that $(i-1,j-1)$ also belongs to the diagram. Then the triple rule implies that $\rtau_{(i,j)}>\rtau_{(i-1,j-1)}$, as $\rtau \in E_{\alpha}$.
Now consider the filling obtained by shifting row $i$ by $i-1$ cells to the left. The resulting filling has shifted reverse shape $\alpha$ and satisfies the following two conditions.
 \begin{enumerate}
 \item The entries along every row decrease from left to right.
 \item The entries along every column increase from top to bottom.
 \end{enumerate}
 Thus, the resulting filling is a shifted reverse tableau of shape $\alpha$, and it is easily seen that it is uniquely determined by $\rtau$.
 
We can also invert the above procedure. Starting from a shifted reverse tableau of shape $\alpha$, consider the filling obtained by shifting row $i$ by $i-1$ cells to the right. We claim that the resulting filling is an SRCT that belongs to $E_{\alpha}$. To see this, first notice that the entry in a fixed cell in the shifted reverse tableau is strictly greater than the entry in any cell that lies weakly north-east of it. Thus, when row $i$ is shifted $i-1$ cells to the right, we are guaranteed that the entries in every column in the resulting filling (which is of reverse composition shape $\alpha$) increase from top to bottom. The entries in every row decrease from left to right as they did in the shifted reverse tableau. The fact that there are no violations of the triple rule also follows from the observation made earlier in the paragraph. Thus, we obtain an element of $E_{\alpha}$ indeed, and this finishes the proof.
 \end{proof}
 
Repeating the same procedure as outlined in the proof above allows us to prove the following statements too.
 
 \begin{proposition}\label{prop:rectangle}
Let $n,k$ be positive integers and $\alpha=(n^k)$, that is, $\alpha$ consists of $k$ parts equal to $n$. Then the number of SRCTs that belong to $E_{\alpha}$, that is $\dim \smodule _{\alpha, E_\alpha }$, is equal to the number of truncated shifted reverse tableaux of shape $\delta_{[n+k-1,n]}\backslash \delta_{k-1}$.
 \end{proposition}
 
 \begin{proposition}\label{prop:deltann}
Let $n$ be a positive integer and $\alpha=(1,2,\ldots,n-1,n,n)$. Then, the number of SRCTs that belong to $E_{\alpha}$, that is $\dim \smodule _{\alpha, E_\alpha }$, is equal to the number of truncated shifted reverse tableaux of shape $\delta_{n+1}\backslash \delta_{1}$.
 \end{proposition}
 
 In \cite[Theorem 4]{panova}, the number of truncated shifted reverse tableaux of shape $\delta_{n+1}\backslash \delta_{1}$ has already been computed in a closed form. This allows us to state the following corollary.
 
 \begin{corollary}\label{cor:deltannclosed}
Given a positive integer $n$, the cardinality of $E_{(1,2,\ldots, n,n)}$ is $g_{n+1}\frac{C_{n+1}C_{n-1}}{2C_{2n-1}}$ where $g_{n+1}=\frac{\binom{n+2}{2} !}{\prod_{0\leq i< j\leq n+1}(i+j)}$ is the number of shifted reverse tableaux of shape $(1,2,\ldots, n+1)$ and $C_{m}=\frac{1}{m+1}\binom{2m}{m}$ is the $m$-th Catalan number. 
 \end{corollary}
 
 In return, we can also calculate the number of certain truncated shifted reverse tableaux from the following.
 
  \begin{theorem}\label{the:3k}
 Let $k$ be a positive integer and $\alpha=(3^k)$, that is, $\alpha$ consists of $k$ parts equal to $3$. Then the number of SRCTs that belong to $E_{\alpha}$, that is $\dim \smodule _{\alpha, E_\alpha }$, is equal to $2^{k-1}$.
 \end{theorem} 
 
 \begin{proof}
Note $E_{\alpha}$ is cyclically generated by $\rtau_{\alpha}$ as it is the source tableau in $E_{\alpha}$ by Corollary~\ref{cor:equivcyclic}. Let $\sigma \in \sgrp_{3k}$ be defined as $\sigma = s_{3k-3}s_{3k-6}\cdots s_{3}$. Also, note that $\rtau_1=\pi_\sigma(\rtau_{\alpha})$ is the unique sink tableau in $E_{\alpha}$. To see this note that $\rtau_1$ is as shown below.

$$\begin{array}{c}
\tiny{\tableau{
4&2&1\\7&5&3\\10&8&6}}\\
\vdots
\\
\tiny{\tableau{3k{-}2 & 3k{-}4 & 3k{-}6\\3k & 3k{-}1 & 3k{-}3}}\end{array}$$

 Thus, all integers $i\equiv 1,2 \pmod{3}$ where $2\leq i\leq 3k-2$ are attacking descents and there are no other descents. This in turn implies that $\rtau_1$ is a sink tableau by definition.
 
 Now for any subset $X=\{i_1<i_2<\cdots<i_j\}$ of $\{3,6,\ldots,3k-3\}$, we can associate a permutation $\omega_{X}=s_{i_j}\cdots s_{i_1}$. Since the reduced word expression for $\omega_X$ is a suffix of some reduced word for $\sigma$, we know that $\pi_{\omega_{X}}(\rtau_{\alpha})$ is an SRCT that belongs to $E_{\alpha}$. In fact, it is not hard to see that all elements of $E_{\alpha}$ are obtained by picking such a subset $X$, since the set of operators $\{\pi_{3r}\}_{r=1}^{k-1}$ is a pairwise commuting set of operators. Thus, $\vert E_{\alpha}\vert= 2^{k-1}$.
  \end{proof}

Hence by this theorem and Proposition~\ref{prop:rectangle} we immediately obtain the following enumerative result.
  
  \begin{corollary}\label{cor:3k}
 Given a positive integer $k$, the number of truncated shifted reverse tableaux of shape $\delta_{[k+2,3]}\backslash \delta_{k-1}$ equals $2^{k-1}$.
  \end{corollary}
  
\section{Restriction rules and skew quasisymmetric Schur functions}\label{sec:restriction}
We now turn our attention to skew quasisymmetric Schur functions and discover that they, too, arise as quasisymmetric characteristics.
To begin we need some definitions to lead us to the definition of skew quasisymmetric Schur functions that first arose in \cite{BLvW}. 

\begin{definition}\label{def:Lc}The \emph{reverse composition poset} $(\Lc, \lessc) $ is the poset consisting of the set of all compositions $\Lc$ in which the composition
 $\alpha=(\alpha_1,\ldots,\alpha_\ell)$ is covered by
\begin{enumerate}
\item $(1, \alpha_1,\ldots,\alpha_\ell)$, that is, the composition obtained by \emph{prefixing} a part of size 1 to $\alpha$.
\item $(\alpha_1,\ldots,\alpha_k+1,\dots,\alpha_\ell)$, provided that $\alpha_i\neq\alpha_k$ for all $i<k$, that is, the composition obtained by adding 1 to a part of $\alpha$ as long as that part is the \emph{leftmost} part of that size.
\end{enumerate}
\end{definition}

\begin{example}\label{ex:Lcchain}
$$(1)\coverc (2)\coverc (1,2)\coverc (1,1,2) \coverc (1,1,3)\coverc (2,1,3)
$$\end{example}

Let $\alpha, \beta$ be two reverse composition diagrams such that $\beta \lessc \alpha$. Then we define the \emph{skew reverse composition shape} $\alpha \cskew \beta$ to be the array of cells of $\alpha$
$$\alpha \cskew \beta =\{  (i,j) \in \alpha \suchthat (i,j) \mbox{ is not in the subdiagram }  \beta\}$$where $\beta$ has been drawn in the bottom left corner of $\alpha$, due to $\lessc$. We refer to $\beta$ as the \emph{inner shape} and to $\alpha$ as the \emph{outer shape}. The \emph{size} of $\alpha \cskew \beta$ is $| \alpha \cskew \beta| = |\alpha | - |\beta|$. Note that $\alpha \cskew \varnothing$ is simply the reverse composition diagram $\alpha$. Hence, we write $\alpha$ instead of $\alpha \cskew \varnothing$ and say it is of \emph{straight shape}.

\begin{example}\label{ex:revcompskewshape}In this example the {inner} shape is denoted by cells filled with a $\bullet$.
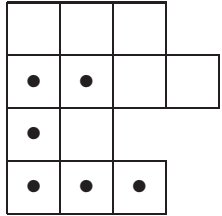
$$\begin{array}{c}
  \tableau{ \ &\ &\  \\
  \bullet & \bullet & \ & \ \\
  \bullet & \ \\
  \bullet & \bullet & \bullet}\\
 \\
\alpha \cskew \beta = (3,4,2,3)\cskew (2,1,3)\end{array}$$
\end{example}

\begin{definition}\cite[Definition 2.9]{BLvW}\label{def:skewSRCT} Given a skew reverse composition shape $\alpha \cskew \beta$ of size $n$, we define a \emph{skew standard reverse composition tableau} (abbreviated to \emph{skew SRCT}) $\rtau$ of \emph{shape} $\alpha \cskew \beta$ and \emph{size} $n$ to be a bijective filling 
$$\rtau:  \alpha \cskew \beta  \rightarrow \{1, \ldots , n\}$$of the cells $(i,j)$ of the skew reverse composition shape $\alpha \cskew \beta$ such that
\begin{enumerate}
\item the entries in each row are decreasing when read from left to right
\item the entries in the first column are increasing when read from top to bottom
\item \emph{triple rule:} set $\rtau(i,j) = \infty$ for all $(i,j)\in \beta$. If $i<j$ and $(j, k+1)\in \alpha \cskew \beta $ and $\rtau (i,k)> \rtau (j, k+1)$, then $\rtau (i,k+1)$ exists and $\rtau (i,k+1)> \rtau (j, k+1)$.
\end{enumerate}
We denote the set of all skew SRCTs of shape $\alpha \cskew \beta$ by $\SRCT (\alpha \cskew \beta).$
\end{definition}

The \emph{descent set} of a skew SRCT $\rtau$ of size $n$, denoted by $\des(\rtau)$, is $$\des(\rtau) = \{ i \suchthat i+1 \mbox{ appears weakly right of } i\}\subseteq [n-1]$$and the corresponding \emph{descent composition} of $\rtau$ is $\comp(\rtau)=\comp(\des(\rtau)).$

\begin{example}\label{ex:skewSRCT}
We have in the skew SRCT below that $\des (\rtau ) = \{1,3,4 \}$ and $\comp (\rtau ) = (1,2,1,2)$.
$$\rtau = \tableau{
6&4&1\\\bullet&\bullet&5&2\\\bullet&3\\\bullet&\bullet&\bullet
}$$
\end{example}

\begin{definition}\label{def:skewSquasisymmetric}
Let $\alpha\cskew \beta $ be a skew reverse composition shape. Then the
\emph{skew quasisymmetric Schur function} $\qs _{\alpha \cskew \beta}$ is defined by \[\qs_{\alpha \cskew \beta}=\sum _{\rtau \in \SRCT(\alpha \cskew \beta)} F _{\comp (\rtau)}.\]
 \end{definition}

Now we can turn our attention to developing restriction rules. Observe that given compositions $\alpha, \beta$ such that $\beta \lessc \alpha$ with $| \alpha \cskew \beta | = m$ and a skew SRCT $\rtau \in \SRCT(\alpha \cskew \beta)$ the definition of attacking is still valid. Thus, on $\rtau \in \SRCT(\alpha \cskew \beta)$ we can define operators $\pi _i$ for $1\leq i \leq m-1$ as in \eqref{eq:pi}. The operators satisfy Theorem~\ref{the:0heckerels} using the same proof techniques as in Section~\ref{sec:0Heckeaction}. {}From here, we can define the \emph{column word} $\col _{\rtau}$ to be the entries in each column read from top to bottom, where the columns are processed from left to right. As in the proof of Lemma~\ref{lem:precimpliesbruhat}, if $i\in \des (\rtau )$ and is non-attacking, then by definition $i$ occurs to the left of $i+1$ in $\col _{\rtau }$ and hence $s_i \col _{\rtau}$ contains one more inversion than $\col _\rtau$ and so as in Section~\ref{sec:partialorder} we obtain a partial order as follows.

\begin{definition}\label{def:skewpartialorder} Let $\alpha, \beta$ be compositions such that $\beta \lessc \alpha$ with $| \alpha \cskew \beta | = m$. Let $\rtau _1, \rtau _2 \in \SRCT (\alpha \cskew \beta)$. Define a partial order $\preccurlyeq _{\alpha \cskew \beta}$ on $\SRCT(\alpha \cskew \beta)$ by $\rtau _1 \preccurlyeq _{\alpha \cskew \beta} \rtau _2$ if and only if there exists a permutation $\sigma \in \sgrp _m$ such that $\pi _\sigma (\rtau _1) = \rtau _2$.
\end{definition}

We will now extend the above partial order $\preccurlyeq _{\alpha \cskew \beta}$ to a total order $\preccurlyeq _{\alpha \cskew \beta} ^t$ with minimal element $\rtau _1$ and define for a given $\rtau \in \SRCT (\alpha\cskew\beta)$
$$\mathcal{V}_{\rtau} = \spam \{ \rtau _j \in \SRCT (\alpha \cskew \beta) \suchthat \rtau  \preccurlyeq _{\alpha \cskew \beta} ^t \rtau _j\}.$$Following the methods of Section~\ref{sec:0Hecke}, we obtain that $\mathcal{V}_{\rtau}$ is an $H_m(0)$-module, and moreover, the following.

\begin{theorem}\label{the:skewbigone} Let $\alpha, \beta$ be compositions such that $\beta \lessc \alpha$ with $| \alpha \cskew \beta | = m$, and let $\rtau _1 \in \SRCT (\alpha \cskew \beta)$ be the minimal element under the total order $\preccurlyeq _{\alpha \cskew \beta} ^t$. Then $\mathcal{V}_{\rtau_1}$ is an $H_m(0)$-module whose quasisymmetric characteristic is the skew quasisymmetric Schur function $\qs_{\alpha \cskew \beta}$.
\end{theorem}

For convenience we will denote the above $H_m(0)$-module $\mathcal{V}_{\rtau_1}$ by $\smodule_{\alpha \cskew \beta}$, and summarize Theorem~\ref{the:skewbigone} by
\begin{equation}\label{eq:skewqsmodule}ch([\smodule_{\alpha \cskew \beta}])= \qs _{\alpha \cskew \beta}.\end{equation}

For $0\leq m\leq n$, by abuse of notation, let $H_{m,n-m}(0)$ denote the subalgebra of $H_{n}(0)$ generated by
\begin{eqnarray*}
\{\pi_1,\ldots,\pi_{m-1},\pi_{m+1},\ldots,\pi_{n-1}\}.
\end{eqnarray*} 
 This subalgebra is actually isomorphic to $H_{m}(0)\otimes H_{n-m}(0)$. The isomorphism is obtained by mapping the generator $\pi_i$ of $H_{m,n-m}(0)$ where $1\leq i\leq n-1$ and $i\neq m$ as follows.
 \begin{eqnarray}
 \pi_i \longmapsto \left\lbrace\begin{array}{ll}\pi_i\otimes1 & 1\leq i<m\\1\otimes \pi_{i-m}& i>m\end{array}\right. 
 \end{eqnarray}
 Here, $1$ denotes the unit of the 0-Hecke algebra.
 
 Let $\alpha\vDash n$. Let $\smodule_{\alpha}\downarrow^{H_{n}(0)}_{H_{m,n-m}(0)}$ denote $\smodule_{\alpha}$ viewed as an $H_{m,n-m}(0)$-module. The isomorphism from before allows us to think of $\smodule_{\alpha}$ as an $H_{m}(0)\otimes H_{n-m}(0)$-module. Given $\rtau\in \SRCT(\alpha)$, let $\rtau_{\leq m}$ denote the skew SRCT comprising of all cells whose entries are $\leq m$. Furthermore, let $\rtau_{>m}$ denote the SRCT of straight shape comprising of all cells whose entries are $>m$ that have had $m$ subtracted from each entry.
Now, consider the following subsets of $\SRCT(\alpha)$ indexed by compositions of $n-m$ 
\begin{eqnarray*}
X_{\beta}=\{\rtau\in \SRCT(\alpha)\suchthat \text{ the shape of } \rtau_{\leq m} \text{ is } \alpha\cskew\beta\}.
\end{eqnarray*}

Then we have that
\begin{eqnarray}\label{eq:X}
\SRCT(\alpha)=\bigsqcup_{\substack{\beta\vDash n-m\\\beta <_{\check{c}}\alpha}}X_{\beta}
\end{eqnarray} 
where $\bigsqcup$ denotes disjoint union.

For any $\beta\cskew\alpha$ satisfying $\beta\vDash n-m$, let $\smodule_{\alpha,X_{\beta}}$ denote the $\mathbb{C}$-linear span of all tableaux in $X_{\beta}$. Then, we can give $\mathbf{S}_{\alpha,X_{\beta}}$ the structure of an $H_{m}(0)\otimes H_{n-m}(0)$-module by defining an action of the generators as follows.
\begin{eqnarray*}
&1\otimes \pi_{i}(\rtau)&=\pi_{i+m}(\rtau) \text{ where } 1\leq i\leq n-m-1\\
& \pi_{i}\otimes 1(\rtau)&=\pi_{i}(\rtau) \text{ where } 1\leq i\leq m-1
\end{eqnarray*}
 It is easily seen that with the action defined above, $\mathbf{S}_{\alpha,X_{\beta}}$ is an $H_{m}(0)\otimes H_{n-m}(0)$-module.
 
 \begin{proposition}\label{prop:skewiso}
 Let $\alpha \vDash n, \beta\vDash n-m$ such that $\beta\lessc\alpha$. Then the following is an isomorphism of $H_{m}(0)\otimes H_{n-m}(0)$-modules.
 $$
  \smodule_{\alpha,X_{\beta}}\cong \smodule_{\alpha\cskew\beta}\otimes \smodule_{\beta}
$$
 \end{proposition}
 
 \begin{proof}
 Consider the map $\theta:  \smodule_{\alpha,X_{\beta}}\longrightarrow \smodule_{\alpha\cskew\beta}\otimes \smodule_{\beta}$ given by
 \begin{eqnarray*}
 \rtau \longmapsto \rtau_{\leq m}\otimes \rtau_{> m}.
 \end{eqnarray*}
The map is well defined as all $\rtau\in \mathbf{S}_{\alpha,X_{\beta}}$ satisfy the condition that the shape of $\rtau_{\leq m}$ is $\alpha\cskew\beta$, and it is an isomorphism of vector spaces. We will show that 
\begin{eqnarray*}
&(1\otimes \pi_i)(\theta(\rtau))&=\theta(1\otimes\pi_i(\rtau)) \text{ when } 1\leq i\leq n-m-1, \\
&(\pi_i\otimes 1)(\theta(\rtau))&=\theta(\pi_i\otimes 1(\rtau)) \text{ when } 1\leq i\leq m-1.
\end{eqnarray*} 

Let $1\leq i\leq n-m-1$. We have
\begin{eqnarray*}
(1\otimes \pi_i)(\theta(\rtau))&=&1\otimes \pi_i(\rtau_{\leq m}\otimes \rtau_{>m})\nonumber\\&=&\rtau_{\leq m}\otimes \pi_i(\rtau_{>m})\nonumber\\&=&\theta(\pi_{i+m}(\rtau))\nonumber\\&=&\theta(1\otimes \pi_i(\rtau)).
\end{eqnarray*}
The verification that $ (\pi_i\otimes 1)(\theta(\rtau))=\theta(\pi_i\otimes 1(\rtau))$ when $ 1\leq i\leq m-1$ is very similar. 
Thus, we get that the map $\theta$ gives the desired isomorphism of $H_m(0)\otimes H_{n-m}(0)$-modules.
 \end{proof}
 Using the above proposition we can state the following restriction rule, which reflects the coproduct formula for quasisymmetric Schur functions \cite[Theorem 3.5]{BLvW}, which in turn reflects the classical coproduct formula for Schur functions.
 
 \begin{theorem}\label{the:restriction_rule}(Restriction rule)
Let $\alpha \vDash n$. Then the following is an isomorphism of $H_{m}(0)\otimes H_{n-m}(0)$-modules.
 \begin{eqnarray*}
 \smodule_{\alpha}\downarrow^{H_{n}(0)}_{H_{m,n-m}(0)} \cong \bigoplus_{\substack{\beta\vDash n-m\\\beta\lessc\alpha}}\smodule_{\alpha\cskew\beta}\otimes \smodule_{\beta}
 \end{eqnarray*}
 \end{theorem}
 \begin{proof}
 This is immediate once we realize the following isomorphism of $H_{m}(0)\otimes H_{n-m}(0)$-modules.
 \begin{eqnarray*}
 \smodule_{\alpha}\downarrow^{H_{n}(0)}_{H_{m,n-m}(0)}\cong \bigoplus_{\substack{\beta\vDash n-m\\\beta\lessc\alpha}}\smodule_{\alpha,X_\beta}
 \end{eqnarray*}
 \end{proof}
 
 Given $\alpha \vDash n$, let us denote by $\alpha ^-$ any composition obtained by the removal of a removable node from $\alpha$. Also, let us observe that 
 $$H_{1}(0)\otimes H_{n-1}(0)\cong H_{1, n-1}(0) \cong H_{n-1}(0)$$since $H_1(0)\cong \mathbb{C}$. Then as a corollary to Theorem~\ref{the:restriction_rule} we have the following branching rule, analogous to the branching rule related to Schur functions \cite[Theorem 2.8.3]{sagan}.

 \begin{corollary}\label{cor:branching_rule}(Branching rule)
Let $\alpha \vDash n$. 
 \begin{eqnarray*}
 \smodule_{\alpha}\downarrow^{H_{n}(0)}_{H_{n-1}(0)} \cong \bigoplus_{\alpha ^-}\smodule_{\alpha ^-}
 \end{eqnarray*}
 \end{corollary}
 
 \section{Further avenues}\label{sec:further} At this point there are a number of natural avenues to pursue, and we conclude by outlining some of them.
Regarding the new basis of canonical quasisymmetric functions from Section~\ref{sec:canonical_basis} a natural path to investigate is the algebraic and combinatorial properties that this basis possesses. Remaining with the results from this section, we could also classify all compositions $\alpha$ for which the cardinality of the equivalence class $E_\alpha$, that is $\dim \smodule _{\alpha, E_\alpha}$, is equal to the number of truncated shifted reverse tableaux of suitable shape. Additionally, we could apply the results from this question to find new enumerative results for truncated shifted reverse tableaux of suitable shape.

Moving towards the representation theory of the $0$-Hecke algebra, a natural path to investigate is whether there is a reciprocal induction rule with respect to the restriction rule of Theorem~\ref{the:restriction_rule}. Turning lastly to discrete geometry, in the light of Theorem~\ref{the:bruhatordersubinterval}, what subintervals of the Bruhat order arise as posets generated by $(E ,\po)$ for some composition $\alpha$, and equivalence class $E$ under $\equiva$? Which of these subintervals are rank-symmetric or rank-unimodal? We know they are not all rank-symmetric as if $\alpha = (2,4)$ then $(E _\alpha,\po)$ is not rank-symmetric. However, the data computed implies the following conjecture, with which we conclude.

\begin{conjecture}\label{con:rankuni} Given a composition $\alpha$ and equivalence class $E$ under $\equiva$, the poset $(E , \po)$ is rank-unimodal.
\end{conjecture}

%: Bibliography

\def\cprime{$'$}


\begin{thebibliography}{10}

\bibitem{adin-king-roichman}
{\sc R.~Adin, R.~King and Y.~Roichman},
{\em Enumeration of standard {Y}oung tableaux of certain truncated shapes},
Electron. J. Combin.
18 (2011)
Paper 20 14 pp.

\bibitem{aguiar-bergeron-sottile}
{\sc M.~Aguiar, N.~Bergeron and F.~Sottile}, {\em {Combinatorial Hopf algebras
  and generalized Dehn-Sommerville relations}}, Compos. Math. 142 (2006)
 1--30.

\bibitem{AFNT}{\sc J.-C.~Aval, V.~F\'{e}ray, J.-C.~Novelli and J.-Y.~Thibon},
{\em Quasi-symmetric functions as polynomial functions on Young diagrams},
J. Algebraic Combin. 
41 (2015)
669--706.

\bibitem{baker-richter} 
{\sc A.~Baker and B.~Richter}, 
{\em Quasisymmetric functions from a topological point of view},
Math. Scand.
103 (2008)
 208--242.

\bibitem{BBSSZ} 
{\sc C.~Berg, N.~Bergeron, F.~Saliola, L.~Serrano and M.~Zabrocki},
{\em Indecomposable modules for the dual immaculate basis of quasi-symmetric functions},
Proc. Amer. Math. Soc.
143 (2015)
991--1000.

\bibitem{berg-serrano}
{\sc C.~Berg and L.~Serrano},
{\em Quasi-symmetric and non-commutative affine Schur functions},
\url{	arXiv:1205.2331}

\bibitem{BLvW} 
{\sc C.~Bessenrodt, K.~Luoto and S.~van Willigenburg}, 
{\em Skew quasisymmetric Schur functions and noncommutative Schur functions},
Adv. Math.
226 (2011)
 4492--4532.

\bibitem{billera-brenti}
{\sc L.~Billera and F.~Brenti}, {\em {Quasisymmetric functions and
  {K}azhdan-{L}usztig polynomials}}, 
  Israel Jour. Math.
   184 (2011)
    317--348. 


\bibitem{billera-jia-reiner}
{\sc L.~Billera, N.~Jia and V.~Reiner}, 
{\em {A quasisymmetric function for
  matroids}}, 
  European J. Combin. 30 (2009) 1727--1757.
  
  \bibitem{HDL}
{\sc L.~Billera, H.~Thomas and S.~van Willigenburg}, {\em {Decomposable
  compositions, symmetric quasisymmetric functions and equality of ribbon Schur
  functions}}, Adv. Math. 204 (2006) 204--240.
  
\bibitem{bjorner}
{\sc A.~Bj{\"o}rner},
{\em Orderings of {C}oxeter groups},
      \newblock {Combinatorics and algebra, Proc. Conf., Boulder/Colo. 1983, Contemp.
  Math. 34 (1984) 175--195}.
 
 \bibitem{bjorner-brenti}
 {\sc A.~Bj{\"o}rner and F.~Brenti},
 {\em Combinatorics of Coxeter Groups}, Springer, 2005.
 
  
  \bibitem{carter}
 {\sc R.~Carter},
{\em Representation theory of the {$0$}-{H}ecke algebra},
J. Algebra
 104 (1986)
 89--103.


\bibitem{cauchy}
{\sc A.~Cauchy}, {\em {M\'{e}moire sur les fonctions qui ne peuvent obtenir que deux valeaurs \'{e}gale et de signes contraires par suite des transpositions op\'{e}r\'{e}s entre les variables qu'elles renferment}}, 
J. \'Ecole Polyt. 10 (1815)
   29--112.
  
  \bibitem{chow} 
{\sc T.~Chow},
{\em Descents, quasi-symmetric functions, Robinson-Schensted for posets, and the chromatic symmetric function}, 
J. Algebraic Combin.
10 (1999)
 227--240.

  
  \bibitem{derksen} 
{\sc H.~Derksen},
{\em Symmetric and quasi-symmetric functions associated to polymatroids},
J. Algebraic Combin.
30 (2009)
 43--86.

\bibitem{DKLT} 
{\sc G.~Duchamp, D.~Krob, B.~Leclerc and J-Y.~Thibon}, 
{\em Fonctions quasi-sym\'etriques, 
fonctions sym\'etriques non-commutatives, et alg\`ebres de Hecke \`a $q = 0$},
C. R. Math. Acad. Sci. Paris
 322 (1996) 
  107--112. 

  
  \bibitem{ehrenborg-1}
{\sc R.~Ehrenborg}, 
{\em {On posets and Hopf algebras}}, 
Adv. Math. 119
  (1996) 1--25.
  
  \bibitem{ferreira} 
{\sc J. Ferreira},
{\em A Littlewood-Richardson type rule for row-strict quasisymmetric Schur functions},
\url{arXiv:1102.1458}

\bibitem{foissy} 
{ \sc L.~Foissy}, 
{\em Fa\`{a} di Bruno subalgebras of the Hopf algebra of planar trees from combinatorial Dyson-Schwinger equations},
Adv. Math.
218 (2008)
 136--162.

  
  \bibitem{garsia-reut}
{\sc A.~Garsia and C.~Reutenauer}, 
{\em A decomposition of {S}olomon's descent
  algebra}, 
  Adv. Math.
   77 (1989)
 189--262.
  
  \bibitem{GKLLRT}
{\sc I.~Gelfand, D.~Krob, A.~Lascoux, B.~Leclerc, V.~Retakh and J.-Y. Thibon},
  {\em {Noncommutative symmetric functions}}, 
  Adv. Math. 112 (1995)
 218--348.

\bibitem{gessel}
{\sc I.~Gessel}, {\em {Multipartite P-partitions and inner products of skew
  Schur functions}}.
\newblock {Combinatorics and algebra, Proc. Conf., Boulder/Colo. 1983, Contemp.
  Math. 34 (1984) 289--301}.
  
  \bibitem{gessel-reut} 
{\sc I.~Gessel and C.~Reutenauer},
{\em Counting permutations with given cycle structure and descent set},
J. Combin. Theory Ser. A
64 (1993)
 189--215.


\bibitem{gnedin-olshanski} 
{\sc A.~Gnedin and G.~Olshanski}, 
{\em Coherent permutations with descent statistic and the boundary problem for the graph of zigzag diagrams},
Int. Math. Res. Not.
  (2006)
Art. ID 51968 39 pp.

 \bibitem{hersh-hsiao} 
{\sc P.~Hersh and S.~Hsiao},
{\em Random walks on quasisymmetric functions},
Adv. Math.
222 (2009)
 782--808.

\bibitem{HHL-1}
{\sc J.~Haglund, M.~Haiman and N.~Loehr}, {\em {A combinatorial formula for
  Macdonald polynomials}},
   J. Amer. Math. Soc. 18 (2005) 735--761.

\bibitem{HLMvW-1}
{\sc J.~Haglund, K.~Luoto, S.~Mason and S.~van Willigenburg}, 
{\em {Quasisymmetric {S}chur functions}}, J. Combin. Theory Ser. A 118 (2011) 
463--490. 
 
  
  \bibitem{hivert}
{\sc F.~Hivert}, 
{\em {Hecke algebras, difference operators, and
  quasi-symmetric functions}}, 
  Adv. Math. 
  155 (2000) 181--238.
  
 

\bibitem{huang-1} 
{\sc J.~Huang},
{\em  $0$-{H}ecke algebra actions on coinvariants and flags},
J. Algebraic Combin. 
40 (2014)  
245--278.

 \bibitem{huang-2} 
{\sc J.~Huang},
{\em $0$-{H}ecke algebra action on the {S}tanley-{R}eisner ring of the {B}oolean algebra},
Ann. Comb. 
19 (2015)
293--323.
  
  \bibitem{humpert} 
{\sc B.~Humpert},
{\em A quasisymmetric function generalization of the chromatic symmetric function}, 
Electron. J. Combin.
18 (2011)
Paper 31 13 pp.

\bibitem{jacobson}
{\sc N.~Jacobson},
{\em Basic algebra. {II}},
W. H. Freeman and Company, 1989.
  
  \bibitem{kwon} 
{\sc J.-H.~Kwon}, 
{\em Crystal graphs for general linear Lie superalgebras and quasi-symmetric functions},
J. Combin. Theory Ser. A
116 (2009)
 1199--1218.

  
  \bibitem{lapointe-morse}
{\sc L.~Lapointe and J.~Morse}, 
{\em A k-tableau characterization of k-{S}chur
  functions}, 
  Adv. Math. 213 (2007)  183--204.
  
  \bibitem{lascoux-1}
{\sc A.~Lascoux}, {\em Puissances ext\'erieures, d\'eterminants et cycles de
  {S}chubert}, Bull. Soc. Math. France 102 (1974) 161--179.
  
  \bibitem{lauve-mason} 
  {\sc A.~Lauve and S.~Mason},
{\em $\Qsym$ over $\Sym$ has a stable basis},
J. Combin. Theory Ser. A
118 (2011)
 1661--1673.

  
  \bibitem{luoto} 
{\sc K. Luoto},  
{\em A matroid-friendly basis for the quasisymmetric functions},
J. Combin. Theory Ser. A
115 (2008)
 777--798.

\bibitem{LMvW}
{\sc K.~Luoto, S.~Mykytiuk and S.~van Willigenburg},
{\em An introduction to quasisymmetric Schur functions - Hopf algebras, quasisymmetric functions, and Young composition tableaux}, Springer, 2013.

  
  \bibitem{macdonald-polys} 
  {\sc I.~Macdonald}, {\em A new class of symmetric functions}, S\'{e}m. Lothar. Combin.
 372 (1988) 131--171.
  
  \bibitem{macdonald-1}
{\sc I.~Macdonald}, {\em {Symmetric functions and Hall polynomials. 2nd ed.}},
  {Oxford University Press}, 1998.
  
  \bibitem{mason-remmel} 
{\sc S.~Mason and  J.~Remmel},
{\em Row-strict quasisymmetric Schur functions},
Ann. Comb. 18 (2014) 127--148.

\bibitem{mathas}
{\sc A. Mathas},
{\em Iwahori-{H}ecke algebras and {S}chur algebras of the symmetric
              group},
American Mathematical Society, 1999.
  
  \bibitem{molev-sagan}
{\sc A.~Molev and B.~Sagan}, 
{\em A {L}ittlewood-{R}ichardson rule for
  factorial {S}chur functions}, 
  Trans. Amer. Math. Soc. 351 (1999)
 4429--4443.
  
  \bibitem{norton} 
  {\sc P.~Norton}, {\em $0$-Hecke algebras}, 
  J. Aust. Math. Soc. 
  27 (1979)
337--357.
  
  \bibitem{okounkov-1}
{\sc A.~Okun{\cprime}kov and G.~Ol{\cprime}shanski{\u\i}}, 
{\em Shifted {S}chur
  functions}, 
  Algebra i Analiz 9 (1997) 73--146.
  
  \bibitem{panova}
  {\sc G.~Panova}
  {\em Tableaux and plane partitions of truncated shapes},
   Adv. in Appl. Math.
   49 (2012)
 196--217.
  
  \bibitem{sagan}
{\sc B.~Sagan}, {\em {The symmetric group. Representations, combinatorial
  algorithms, and symmetric functions. 2nd ed.}}, Springer, 2001.
  
  \bibitem{schur}
{\sc I.~Schur}, {\em \"{U}ber eine {K}lasse von {M}atrizen, die sich einer
  gegebenen {M}atrix zuordnen lassen}.
\newblock Dissertation, Berlin, 1901.
\newblock In: I.~Schur, Gesammelte Abhandlungen I, Springer, Berlin,
  (1973) 1--70.
  
    \bibitem{solomon}
{\sc L.~Solomon}, {\em A {M}ackey formula in the group ring of a {C}oxeter
  group}, J. Algebra 41 (1976) 255--264.
  
  \bibitem{stanley-chrom} 
{\sc R.~Stanley},
{\em A symmetric function generalization of the chromatic polynomial of a graph},
Adv. Math.
111 (1995)
 166--194.

  
  \bibitem{stanley-ec2}
{\sc R.~Stanley}, {\em {Enumerative
  Combinatorics} vol.~2}, Cambridge University Press, 1999.
  
  \bibitem{stanley-riffle}
{\sc R.~Stanley}, {\em Generalized riffle
  shuffles and quasisymmetric functions}, Ann. Comb. 5 (2001) 479--491.
  
  \bibitem{stembridge-1}
{\sc J.~Stembridge}, {\em Shifted tableaux and the projective representations
  of symmetric groups}, Adv. Math. 74 (1989)  87--134.
  
\bibitem{VasuMN}
{\sc V. Tewari},
{\em
A {M}urnaghan-{N}akayama rule for noncommutative Schur functions},
\url{arXiv: 1403.0607}


\end{thebibliography}
\end{document}